\documentclass[11pt]{amsart}

\usepackage{epigamath}

\usepackage[english]{babel}

\numberwithin{equation}{section}

\numberwithin{table}{section}
\numberwithin{figure}{section}

\usepackage[shortlabels]{enumitem}
\setlist[enumerate,1]{label={\rm(\arabic*)}, ref={\rm\arabic*}} 

\usepackage{amsmath}
\usepackage{tikz-cd}
\usepackage{xypic}
\usepackage{subcaption}
\DeclareCaptionSubType*{figure}

\newtheorem{theorem}[equation]{Theorem}
\newtheorem{proposition}[equation]{Proposition}

\newtheorem{lemma}[equation]{Lemma}

\theoremstyle{definition}
\newtheorem{definition}[equation]{Definition}

\theoremstyle{remark}
\newtheorem{remark}[equation]{Remark}
\newtheorem{example}[equation]{Example}
\usepackage{thmtools, thm-restate}

\newcommand{\sC}{\mathcal{C}}

\newcommand{\sF}{\mathcal{F}}
\newcommand{\sG}{\mathcal{G}}

\newcommand{\sO}{\mathcal{O}}
\newcommand{\sP}{\mathcal{P}}

\renewcommand{\AA}{\mathbb{A}}

\newcommand{\PP}{\mathbb{P}}
\newcommand{\QQ}{\mathbb{Q}}
\newcommand{\ZZ}{\mathbb{Z}}
\newcommand{\NN}{\mathbb{N}}
\newcommand{\RR}{\mathbb{R}}

\newcommand{\lpt}{\mathfrak{S}}

\usepackage{bm}
\newcommand{\lc}{\bm{[}}
\newcommand{\rc}{\bm{]}}
\newcommand{\ro}{\bm{)}}
\newcommand{\lo}{\bm{(}}

\newcommand{\chiloc}{\chi_\mathrm{loc}}

\DeclareMathOperator{\Spec}{Spec}

\DeclareMathOperator{\Conv}{Conv}
\DeclareMathOperator{\Pos}{Pos}
\DeclareMathOperator{\vol}{vol}
\DeclareMathOperator{\ord}{ord}
\DeclareMathOperator{\Hom}{Hom}
\DeclareMathOperator{\SL}{SL}

\renewcommand{\H}{\operatorname{H}}

\renewcommand{\k}{\mathbf{k}}

\newcommand{\lra}{\longrightarrow}
\newcommand{\supth}[1]{\ensuremath{#1^{\mathrm{th}}}}
\newcommand{\supst}[1]{\ensuremath{#1^{\mathrm{st}}}}

\EpigaVolumeYear{9}{2025} \EpigaArticleNr{2} \ReceivedOn{December 7, 2023}
\InFinalFormOn{April 18, 2024}
\AcceptedOn{May 24, 2024}

\title{Local Euler characteristics of $\boldsymbol{A_n}$-singularities and their application to hyperbolicity}
\titlemark{Local Euler characteristics of $\boldsymbol{A_n}$-singularities}

\author{Nils Bruin}
\address{Department of Mathematics, Simon Fraser University, Burnaby, BC, Canada V5A 1S6}
\email{nbruin@sfu.ca}
\author{Nathan Ilten}
\address{Department of Mathematics, Simon Fraser University, Burnaby, BC, Canada V5A 1S6}
\email{nilten@sfu.ca}
\author{Zhe Xu}
\address{Department of Mathematics, University of Oregon, Eugene, OR 97403, USA}
\email{zhxu@uoregon.edu}

\authormark{N.~Bruin, N.~Ilten, and Z.~Xu}

\AbstractInEnglish{Wahl's local Euler characteristic measures the local contributions of a singularity to the usual Euler characteristic of a sheaf. Using tools from toric geometry, we study the local Euler characteristic of sheaves of symmetric differentials for isolated surface singularities of type $A_n$. We prove an explicit formula for the local Euler characteristic of the $\supth{m}$ symmetric power of the cotangent bundle; this is a quasi-polynomial in $m$ of period $n+1$. We also express the components of the local Euler characteristic as a count of lattice points in a non-convex polyhedron, again showing it is a quasi-polynomial. We apply our computations to obtain new examples of algebraic quasi-hyperbolic surfaces in $\PP^3$ of low degree.
We show that an explicit family of surfaces with many singularities constructed by Labs has no genus $0$ curves for the members of degree at least $8$ and no curves of genus $0$ or $1$ for degree at least $10$.} 


\begin{document}


\maketitle

\begin{prelims}

\DisplayAbstractInEnglish

\bigskip

\DisplayKeyWords

\medskip

\DisplayMSCclass

\end{prelims}


\newpage

\setcounter{tocdepth}{1}

\tableofcontents


\section{Introduction}
\subsection{Motivation and local Euler characteristic}\label{sec:intro}
We present a full analysis of the local Euler characteristic of the cotangent sheaf at a surface singularity of type $A_n$. Our main motivation is its application to a method for proving that certain surfaces are algebraically quasi-hyperbolic by showing that sufficiently high symmetric powers of the cotangent sheaf have global sections.

Let $Y$ be a non-singular projective surface over a field $\k$ of characteristic $0$. For simplicity we assume that $\k$ is algebraically closed. We say that $Y$ is \emph{algebraically quasi-hyperbolic} if it contains only finitely many curves of genus $0$ and $1$. Coskun and Riedl \cite{CoskunRiedl2023} proved that very general surfaces in $\PP^3$ of degree $d\geq 5$ are algebraically hyperbolic, which is a property that implies they are algebraically quasi-hyperbolic as well. However, no surface defined over a number field is very general, so for many specific surfaces, the question about their quasi-hyperbolicity remains open.

For surfaces of general type, Bogomolov \cite{Bogomolov} shows that if the cotangent bundle on $Y$ is big, then $Y$ is algebraically quasi-hyperbolic. 

The cotangent bundle of a non-singular surface $X\subset\PP^3$ is never big.
However, Bogomolov and de Oliveira \cite{BO} observed that the resolution $\phi\colon Y\to X$ of a normal surface $X$ may have a big cotangent bundle if~$X$ has sufficiently many singularities for its degree. One way to see this is by considering the $\supth{m}$ symmetric power $\sF=S^m\Omega_Y$ of the cotangent sheaf on $Y$. This is a vector bundle, and in particular reflexive. We take the direct image of $\sF$ on $X$ and take its reflexive hull $\sF'$. We have $\sF'=\hat{S}^m \Omega_X$, where $\hat{S}^m\Omega_X$ denotes the reflexive hull of $S^m\Omega_X$. 

As Blache \cite[Section~3.9]{Blache} shows, if the singular locus $S$ of $X$ consists of ADE-singularities, then \emph{local Euler characteristics} as defined by Wahl in~\cite{Wahl1976} can be used to express the difference in Euler characteristics as a sum of local contributions at the singularities $s\in S$ as
\begin{equation}\label{eqn:blache}
	\chi(X,{\sF'})=\chi(Y,\sF)+\sum_{s\in S} \chiloc(s,\sF),
\end{equation}
where $\chiloc(s,\sF)$ is defined as follows.
For a sufficiently small open affine neighbourhood $X^\circ$ of $s$, together with $Y^\circ=\phi^{-1}X^\circ$ and $E_s=\phi^{-1}(s)$, we set
\[
\begin{aligned}
\chiloc(s,\sF)&=\chi^0(s,\sF)+\chi^1(s,\sF), \text{ where }\\
\chi^0(s,\sF)&=\dim \H^0(Y^\circ-E_s,\sF)/\H^0(Y^\circ,\sF)\text{ and}\\
\chi^1(s,\sF)&=\dim \H^1(Y^\circ,\sF).
\end{aligned}
\]
We note that $\H^0(Y^\circ-E_s,\sF)\simeq \H^0(X^\circ-s,\sF')$ and that, thanks to reflexivity, $\H^0(X^\circ-s,\sF')\simeq \H^0(X^\circ,\sF')$. Hence, $\chi^0(s, S^m\Omega_Y)$ gives a bound on the number of conditions that sections of $\sF'$ need to satisfy to extend to regular sections on $E_s$ upon pull-back. Globally, this yields
\begin{equation*}
h^0(Y,S^m\Omega_Y)\geq h^0(X,\hat{S}^m \Omega_X)-\sum_s \chi^0(s, S^m\Omega_Y).
\end{equation*}
For $Y$ of general type, we have $h^2(X,\hat{S}^m\Omega_X)=0$ for $m>2$ by \cite[Proposition~2.3]{BO}, which implies that $h^0(X,\hat{S}^m \Omega_X)\geq \chi(X,\hat{S}^m\Omega_X)$ for $m\geq 3$. We obtain
\begin{equation}
\label{eq:bound_chi1}
h^0(Y,S^m\Omega_Y)\geq \chi(Y,S^m\Omega_Y)+\sum_s \chi^1(s,S^m\Omega_Y)\quad\text{for }m\geq 3.
\end{equation}
By \cite[Proposition~3.7]{BTA2022} we have $\chi^1(s,S^m\Omega_Y)=\frac{4}{27}m^3+O(m)$ for an $A_1$-singularity. It follows (see \cite[Example~4.2 and Remark~4.3]{BTA2022}) that a surface $X\subset \PP^3$ of degree $d\geq 5$ with $r>\frac{9}{4}(2d^2-5d)$ singularities of type $A_1$ has a big cotangent bundle.

\subsection{Local Euler characteristics at $\boldsymbol{A_n}$-singularities}

We compute the local Euler characteristic and its components for an $A_n$-singularity $s_n$. Specifically, we prove the following in Section~\ref{proof:chiloc}.

\begin{theorem}\label{T:chiloc}
	For an $A_n$-singularity $s_n$ on a surface $X$ with minimal resolution $Y\to X$, we have
	\[
	\chiloc(s_n,S^m\Omega^1_Y)=\frac{(n+1)^2-1}{(n+1)}\left(\frac{1}{6}m^3+\frac{1}{2}m^2+\frac{1}{4}m\right)+\frac{b_{n}(m)}{4(n+1)}\cdot m+\frac{c_n(m)}{12(n+1)}, 
	\]
	where
	\begin{align*}
		b_n(m)&=\begin{cases}
			0 &  n\ \textrm{even},\\
			1 &  n\ \textrm{odd}, q\ \textrm{even},\\
			-1 &  n\ \textrm{odd}, q\ \textrm{odd},\\
		\end{cases}
		\\
		c_n(m)&=\begin{cases}
			2q^3-3(n-1)q^2+(n^2-4n-2)q &n\ \textrm{even}, q\ \textrm{even},\\
			2q^3-3(n-1)q^2+(n^2-4n-2)q-3(n+1) &n\ \textrm{even}, q\ \textrm{odd},\\
			2q^3-3(n-1)q^2+(n^2-4n-5)q &n\ \textrm{odd}, q\ \textrm{even},\\
			2q^3-3(n-1)q^2+(n^2-4n+1)q-3(n+1) &n\ \textrm{odd}, q\ \textrm{odd},\\
		\end{cases}
	\end{align*}
	and $q$ is the remainder of $m$ divided by $n+1$.
\end{theorem}
We also compute $\chi^0(s_n,S^m\Omega^1_Y)$ as a sum of lattice point counts in rational polytopes. In order to formulate the result, we need to define some vertices. Consider
\[\begin{aligned}
	P_i&=\left(-\frac{1}{i+1},0,0\right), \quad
	Q_i=\left(-\frac{2}{(i+1)(i+2)},-\frac{i}{i+2},\frac{i}{i+2}\right)\quad\text{for } i=0,1,\ldots,\\
	Z&=(0,-1,0),\\
	P_n'&=\left(\frac{1}{n+1},-1,0\right),\quad
	Q_n'=\left(\frac{2}{(n+1)(n+2)},-1,\frac{n}{n+2}\right).
\end{aligned}\]
Writing $\Conv V$ for the convex hull of a set $V$, we consider the half-open convex polytopes
\begin{equation}\label{E:polyhedra}
	\begin{aligned}
		\sP_i&=\Conv\{P_{i-1},Q_{i-1},P_i,Q_i,Z\}\setminus \Conv\{P_i,Q_i,Z\}\setminus \Conv \{P_{i-1},P_i,Z\},\\
		\sC_n&=\Conv\{P_n,P'_n,Q_n,Q'_n,Z\}\setminus\Conv\{P_n,P_n',Z\}.
	\end{aligned}
\end{equation}
For a polytope $\sP$ we consider the Ehrhart function counting integral point in dilations of the polytope $\sP$,
\[L(\sP,t)=\#\left(t\sP \cap \ZZ^3\right)\quad\text{for }t=0,1,2,\ldots.\]
For a convex polytope spanned by vertices with rational coordinates, this function is a \emph{quasi-polynomial}. The same holds for a non-convex polytope. In Section~\ref{S:proof_chi0} we prove the following.
\begin{theorem}\label{T:chi0}
	Let $Y\to X$ be the minimal resolution of a surface singularity $s_n\in X$ of type $A_n$. Then
	\[\chi^0\left(s_n,S^m\Omega^1_Y\right)=L(\sC_n,m+1)+2\sum_{i=1}^n L(\sP_i,m+1),\]
	which is the lattice point count of the dilations of a half-open non-convex polytope with volume
	\[\vol \sC_n+2\sum_{i=1}^n\vol \sP_i.\]
\end{theorem}
\noindent See the appendix  for the generating functions of  $L(\sP_n,m+1)$ and $L(\sC_n,m+1)$ for small values of $n$.

By considering the non-convex polytope as $n\to\infty$, we obtain some extra information; see Section~\ref{S:proof_chi0_asymptotic} for the proof.

\begin{proposition}\label{P:chi0_asymptotic}\leavevmode
	\begin{enumerate}
		\item The function $\chi^0(s_n,S^m\Omega^1_Y)$ is non-decreasing  both in $n$ and in $m$.
		\item The function $\chi^0(s_n,S^m\Omega^1_Y)$ is constant in $n$ for $n>m$.
		\item For any fixed $n$ we have $\chi^0(s_n,S^m\Omega^1_Y)\leq (\frac{2}{9}\pi^2-2)(m+1)^3+O(m^2)$, where $\pi$ denotes Archimedes' constant, so $\frac{2}{9}\pi^2-2\approx0.1932$.
		
	\end{enumerate}
\end{proposition}

For the proofs of the above results, we employ tools from toric geometry. Consider the affine variety $X\colon x_1x_2=x_3^{n+1}\subset\AA^3$ with its $A_n$-singularity $s=(0,0,0)$, as well as its minimal resolution $Y\to X$. 
Both $X$ and $Y$ are toric varieties; see Examples \ref{ex:An} and \ref{ex:Anres}.

The reflexive hull of the symmetric powers of the cotangent sheaf on $X$,  and the symmetric powers of the cotangent bundle on $Y$, are torus-equivariant reflexive sheaves. In general, for any equivariant reflexive sheaf $\sF$ on a toric variety $Z$, the equivariant structure provides a grading of the cohomology parametrized by the character lattice $M$ of the maximal torus:
\[\H^i(Z,\sF)=\bigoplus_{u\in M}\H^{i}(Z,\sF)_u.\]
Klyachko \cite{Klyachko1989} gives a very detailed description of these graded pieces in terms of combinatorial data associated to $Z$ and $\sF$; this applies in particular to the sheaves $\hat{S}^m\Omega_X$ and $S^m\Omega_Y$.
We can express the quantities in Theorems~\ref{T:chiloc} and \ref{T:chi0} as sums of graded parts as well.
Using Klyachko's machinery we find that only finitely many of these graded parts are non-trivial and that we can express them as lattice point counts in a non-convex polytope dilated by a factor of $m+1$. For Theorem~\ref{T:chiloc} this expression significantly simplifies through the use of lattice-preserving scissor operations and manipulations of generating functions.

\subsection{Applications to algebraic quasi-hyperbolicity}

Comparing the results from Theorem~\ref{T:chiloc} and Proposition~\ref{P:chi0_asymptotic}, we see that the coefficient of $m^3$ in $\chi^0(s_n,S^m\Omega^1_Y)$ is bounded in $n$, whereas in $\chiloc(s_n,S^m\Omega^1_Y)$ it grows linearly with $n$. As a result, we see that the inequality \eqref{eq:bound_chi1} improves as $n$ grows.

A hypersurface $X\subset\PP^3$ of degree $d\geq 5$ with $r$ singularities of type $A_n$ is of general type. For a minimal desingularization $Y$ of $X$, one can compute $\chi(Y,S^m\Omega^1_Y)$ by Atiyah's observation that this equals $\chi(Z,S^m\Omega^1_Z)$ for a smooth degree $d$ surface $Z\subset\PP^3$, combined with a standard application of Hirzebruch--Riemann--Roch and a Chern class computation. This is outlined in \cite{BO}, and the explicit full formulae are given in \cite[Appendix]{BTA2022}. The leading term is
\begin{equation}
\chi\left(Y,S^m\Omega^1_Y\right)= -\tfrac{1}{3}(2d^2-5d)m^3+O\left(m^2\right).
\end{equation}
By combining Theorems~\ref{T:chiloc} and \ref{T:chi0}, we can compute for any particular $n$ the formula for $\chi^1(s_n,S^m\Omega^1_Y)$ as a quasi-polynomial in $m$ and therefore compute the bound \eqref{eq:bound_chi1} explicitly as a function of $m$ and $r$. It is then a matter of simple algebra to determine a bound $r(d,n)$ such that for $r\geq r(d,n)$ we have that $\chi(Y,S^m\Omega^1_Y)$ is a quasi-polynomial of degree $3$ in $m$ with a positive coefficient for $m^3$. This ensures that for large enough $m$ the sheaf $S^m\Omega^1_Y$ has global sections, guaranteeing the algebraic quasi-hyperbolicity; in fact, it guarantees that the cotangent sheaf is big.
We tabulate some values for $r(d,n)$ for small $d,n$ in Table~\ref{table:rdn}. Code to compute the requisite quasi-polynomials is available; see \cite{BIXanc}.
\begin{table}
\[\begin{array}{c|cccccc}
	&n=1&n=2&n=3&n=4&n=5&n=6\\
	\hline
	d=5&57 & 27 & 18 & 13 & 11 & - \\
	d=6&95 & 46 & 30 & 22 & 18 & 15 \\
	d=7&142 & 68 & 45 & 33 & 27 & 22 \\
	d=8&199 & 95 & 62 & 46 & 37 & 31 \\
	d=9&264 & 126 & 83 & 61 & 49 & 41 \\
	d=10&338 & 162 & 106 & 78 & 62 & 52
\end{array}\]
\caption{Values for $r(d,n)$ for small $d,n$.}
\label{table:rdn}
\end{table}

Miyaoka \cite{Miyaoka1984} shows that a degree $d$ surface has at most $\frac{2}{3}(d-1)^2d(n+1)/(2n+1)$ singularities of type $A_n$. Hence, we see that for $n=1$ the smallest realizable degree would be $d=10$, and indeed Barth's decic surface has $r=345$ singularities of type $A_1$ and therefore has big cotangent bundle. For $n\geq 2$ we see that Miyaoka's bound does not exclude any $d$.

Labs \cite[Corollary~A]{Labs2006}, using a construction attributed to Segre \cite{Segre1952} and generalized by Galliarti \cite{Gallarati1952}, describes surfaces of degree $d=2k$ with $4k^2$ singularities of type $A_{k-1}$. An explicit equation (see Section~\ref{S:labs}) for such surfaces is
\[X_k\colon \xi_0^{2k}+\xi_1^{2k}+\xi_2^{2k}+\xi_3^{2k}
-\xi_0^k\xi_1^k-\xi_0^k\xi_2^k-\xi_0^k\xi_3^k-\xi_1^k\xi_2^k-\xi_1^k\xi_3^k-\xi_2^k\xi_3^k=0.\]
For $k\geq 4$ these surfaces have enough singularities to force the cotangent bundle on their minimal resolutions to be big, and hence these surfaces are algebraically quasi-hyperbolic, as was also found by Weiss; see \cite[Corollary~1.1.17]{WeissPhD}. 

While very general surfaces of degree at least $5$ are algebraically hyperbolic by \cite{CoskunRiedl2023}, no surface defined over a number field is very general. The surface $X_4$ is an explicit degree $8$ surface in $\PP^3$ that is algebraically quasi-hyperbolic. To our knowledge this is the lowest-degree explicit example.

We can in fact prove a little more by computing a regular symmetric differential on $X_k$ for $k\geq 4$; see Section~\ref{S:labs_hyperbolic} for the proof.
\begin{theorem}\label{T:labs_hyperbolic}
	For $k\geq 4$ the surface $X_k$ contains no genus $0$ curves. For $k\geq 5$ the surface $X_k$ contains no curves of genus $0$ or $1$.
\end{theorem}

\subsection{Literature}
Bogomolov and de Oliveira \cite{BO} first considered algebraic quasi-hyperbolicity of hypersurfaces with $A_1$-singularities. Due to an error in their computations, they are led to consider an alternative inequality to \eqref{eq:bound_chi1} that is established through Serre duality,
\begin{align}
\label{eq:bound_chi0}h^0\left(Y,S^m\Omega^1_Y\right)&\geq \chi\left(Y,S^m\Omega^1_Y\right)+\sum_s \chi^0\left(s,S^m\Omega^1_Y\right).
\end{align}
Bruin--Thomas--V\'arilly-Alvarado \cite{BTA2022} correct the error and compute $\chi^0(s_1,S^m\Omega^1_Y)$ and $\chi^1(s_1,S^m\Omega^1_Y)$ exactly. They also generalize the results to complete intersection surfaces and give several examples of algebraically quasi-hyperbolic ones.

Using an orbifold approach, Roulleau--Rousseau \cite{RoulleauRousseau} approximate the local Euler characteristic of an $A_n$-singularity $s_n$ by $\chiloc(s_n,S^m\Omega^1_Y)=\frac{n(n+2)}{6(n+1)}m^3+O(m^2)$, consistent with Theorem~\ref{T:chiloc}. They combine equations \eqref{eq:bound_chi1} and \eqref{eq:bound_chi0} to a weaker inequality
\[h^0\left(Y,S^m\Omega^1_Y\right)\geq \chi\left(Y,S^m\Omega^1_Y\right)+\tfrac{1}{2}\sum_s \chiloc\left(s,S^m\Omega^1_Y\right),\]
which allows them to identify examples of degree $d\geq 13$ with sufficient $A_1$-singularities for their bound to imply algebraic quasi-hyperbolicity.
From Proposition~\ref{P:chi0_asymptotic} it follows that \eqref{eq:bound_chi1} gives the stronger result for $n\geq 2$.

De Oliveira--Weiss \cite{OliveiraWeiss2019} consider $A_2$-singularities and reference an approximation to $\chi^0(s_2,S^m\Omega_Y)$ that is consistent with Theorem~\ref{T:chi0}. They also reference \cite{Labs2006} for an example of a degree $9$ surface with sufficiently many $A_2$-singularities to conclude it has big cotangent bundle.
Theorem~\ref{T:chi0} and Proposition~\ref{P:chi0_asymptotic} largely follow the exposition in the third author's master's thesis \cite{XuMSc}. The leading coefficient in $m$ for $\chi^0(s_n,S^m\Omega_Y)$ for $A_n$-singularities is derived independently by Weiss \cite{WeissPhD}, and the top two coefficients are determined independently by Asega--de Oliveira--Weiss \cite{AsegaDeOliveiraWeiss2023}.

Explicit computations with symmetric differentials as in Section~\ref{S:labs_hyperbolic} go back to Vojta \cite{Vojta2000}. See also \cite{BTA2022} for more elaborate examples.

\section{Preliminaries}
\subsection{Toric varieties}
We recall here the necessary basics of toric geometry. See \cite{CLS} for more details. 
Let $N$ be a finitely generated free abelian group with dual $M=\Hom(N,\ZZ)$. Given a pointed polyhedral cone $\sigma\subseteq N\otimes \RR$, its dual is
\[
	\sigma^\vee=\{u\in M\otimes \RR\ |\ \langle v,u\rangle\geq 0\ \forall\ v\in \sigma\}.
\]

Here $\langle v,u \rangle$ is the natural pairing induced by the duality of $N$ and $M$. The semigroup $\sigma^\vee\cap M$ is finitely generated, and
\[
	X_\sigma=\Spec \k[\sigma^\vee\cap M]
\]
is the \emph{affine toric variety} associated to the cone $\sigma$. The dimension of $X_\sigma$ is simply the rank of $N$. The $M$-grading of $\k[\sigma^\vee\cap M]$
 induces an inclusion of the torus
$T=\Spec \k[M]=N\otimes \k$ in $X_\sigma$, with the action of $T$ on itself extending to $X_\sigma$.

\begin{example}[An $A_n$-singularity]\label{ex:An}
	We take $M=N=\ZZ^2$, with $\langle \cdot,\cdot \rangle$ the standard inner product. Let $\sigma_{A_n}$ be the cone generated by $(0,1)$ and $(n+1,1)$. Its dual $\sigma_{A_n}^\vee$ is generated by $(1,0)$ and $(-1,n+1)$. The semigroup $\sigma_{A_n}^\vee\cap M$ is generated by $(1,0)$, $(-1,n+1)$, and $(0,1)$. See Figure \ref{fig:An}. These generators satisfy the relation
	\[
(1,0)+(-1,n+1)=(n+1)\cdot (0,1),
	\]
	so the toric variety $X_{\sigma_{A_n}}$ is isomorphic to the vanishing locus of $x_1x_2-x_3^{n+1}$ in $\AA^3$. This is an isolated surface singularity of type $A_n$.
\end{example}
\begin{figure}
\scriptsize{
\[	\begin{array}{c c}
	\begin{tikzpicture}
		\filldraw[lightgray] (3.9,1.3) -- (0,0) -- (0,1.3);
		\draw[->] (0,0) -- (3.9,1.3);
		\draw[->] (0,0) -- (0,1.3);
\filldraw (0,1) circle [radius=.04];
\filldraw (3,1) circle [radius=.04];
\node[above left] at  (0,1) {$(0,1)$};
\node[below right] at  (3,1) {$(n+1,1)$};
	\end{tikzpicture} &\qquad\qquad
	\begin{tikzpicture}
		\filldraw[lightgray] (1.3,0) -- (0,0) -- (-1.1,3.3);
		\draw[->] (0,0) -- (1.3,0);
		\draw[->] (0,0) -- (-1.1,3.3);
\filldraw (1,0) circle [radius=.04];
\filldraw (-1,3) circle [radius=.04];
\filldraw (0,1) circle [radius=.04];
\node[below left] at  (-1,3) {$(-1,n+1)$};
\node[right] at  (0,1) {$(0,1)$};
\node[below right] at  (1,0) {$(1,0)$};
	\end{tikzpicture}\\
	\sigma_{A_n}&\sigma_{A_n}^\vee
\end{array}
\]
}	\caption{The cone and dual cone for an $A_n$-singularity}\label{fig:An}
\end{figure}
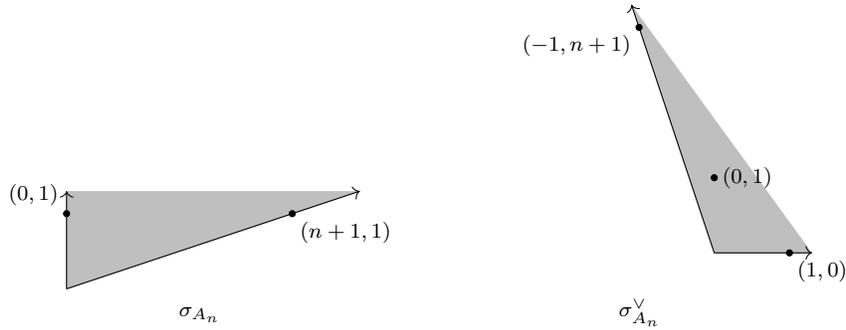

The above construction globalizes. Let $\Sigma$ be a \emph{fan} in $N\otimes \RR$, that is, a collection of pointed polyhedral cones that is closed under taking faces, and such that any two elements intersect in a common face. Any face relation $\tau\prec \sigma$ for $\sigma\in\Sigma$ induces an open inclusion $X_\tau\hookrightarrow X_\sigma$. The toric variety $X_\Sigma$ is constructed by gluing together the affine toric varieties 
\[
	\left\{X_\sigma\right\}_{\sigma\in\Sigma}
\]
along the open immersions induced by face relations; see \cite[Section~3.1]{CLS} for precise details. Moreover, any normal variety $X$ equipped with an effective action of the torus $T$ can be constructed in this fashion; see  \cite[Corollary 3.1.8]{CLS}.

Many aspects of the geometry of $X_\Sigma$ can be read directly from $\Sigma$. For instance, $X_\Sigma$ is non-singular if and only if the fan $\Sigma$ is \emph{smooth}, that is, the primitive lattice generators for each cone in $\Sigma$ can be completed to a basis of $N$; see \cite[Theorem 3.1.19]{CLS}.
For any natural number $i$, let $\Sigma^{(i)}$ be the set of $i$-dimensional cones in $\Sigma$.
Torus-invariant prime divisors on $X_\Sigma$ are in bijection with elements of $\Sigma^{(1)}$; see \cite[Section~4.1]{CLS}.
Given a ray $\rho\in\Sigma^{(1)}$, we denote the corresponding prime divisor by $D_\rho$.
We will denote the primitive lattice generator of the ray $\rho$ by $\nu_\rho$.
The valuation determined by a divisor $D_\rho$ is easily described:
for any ray $\rho\in\Sigma^{(1)}$ and $u\in M$, we have
\begin{equation}\label{eqn:vanishing}
	\ord_{D_\rho}\left(x^u\right)=\langle \nu_\rho,u\rangle, 
\end{equation}
where $x^u$ is the rational function on the torus corresponding to $u$ and $\ord_{D_\rho}(x^u)$ denotes its order of vanishing along $D_\rho$. 

\begin{example}[The minimal resolution of an $A_n$-singularity]\label{ex:Anres}
	Continuing with $M=N=\ZZ^2$, for $i=0,1,\ldots,n+1$ we let $\rho_i$ be the ray in $N\otimes \RR$ generated by $(i,1)$. Consider the fan $\Sigma$ whose $n+1$ top-dimensional cones are generated by $\rho_i,\rho_{i+1}$ for $i=0,\ldots,n$. See Figure \ref{fig:Anres}.

	The fan $\Sigma$ is smooth, so the resulting surface $X_\Sigma$ is non-singular. In fact, the toric variety $X_\Sigma$ is the minimal resolution of the $A_n$-surface singularity from Example \ref{ex:An}. Indeed, the inclusion of each cone of $\Sigma$ in the cone $\sigma_{A_n}$ generated by $\rho_0,\rho_{n+1}$ induces a birational morphism $\phi:X_\Sigma\to X_\sigma$. The morphism $\phi$ is proper since the union of the cones in $\Sigma$ is just $\sigma_{A_n}$; see \cite[Theorem 3.4.11]{CLS}.

	Since the subfan of $\Sigma$ consisting of $\rho_0$, $\rho_{n+1}$, and the origin is the non-singular locus of $X_{\sigma_{A_n}}$, the exceptional locus $E$ of $\phi$ is the union of the prime divisors $E_1=D_{\rho_1}$, \ldots, $E_n=D_{\rho_n}$. Using \textit{e.g.} \cite[Theorem 10.4.4]{CLS} one computes that each $E_i$ is a $(-2)$-curve, so the resolution $\phi$ is indeed minimal.
\end{example}

\begin{figure}
	\scriptsize{
	\begin{tikzpicture}
		\filldraw[lightgray] (7.8,1.3) -- (0,0) -- (0,1.3);
		\filldraw[lightgray] (7.8,1.3) -- (0,0) -- (0,1.3);
		\draw[->] (0,0) -- (7.8,1.3);
		\draw[->] (0,0) -- (0,1.3);
		\draw[->] (0,0) -- (2.6,1.3);
		\draw[->] (0,0) -- (1.3,1.3);
		\draw[->] (0,0) -- (6.5,1.3);
\filldraw (0,1) circle [radius=.04];
\filldraw (6,1) circle [radius=.04];
\node[above left] at  (0,1) {$(0,1)$};
\node[above right] at  (0,1.3) {$\rho_0$};
\node[above right] at  (1.3,1.3) {$\rho_1$};
\node[above right] at  (2.6,1.3) {$\rho_2$};
\node[above right] at  (6.5,1.3) {$\rho_n$};
\node[above right] at  (7.8,1.3) {$\rho_{n+1}$};
\node[above ] at  (4.5,1.3) {$\cdots\cdots\cdots\cdots$};
\node[below right] at  (6,1) {$(n+1,1)$};
	\end{tikzpicture} 
}
	\caption{The minimal resolution of an $A_n$-singularity}\label{fig:Anres}
\end{figure}
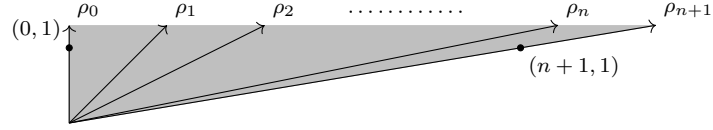

\subsection{Torus-equivariant reflexive sheaves}\label{sec:klyachko}
Let $\sF$ be a $T$-equivariant reflexive sheaf on the toric variety $X_\Sigma$. In \cite{Klyachko1989,klyachko} Klyachko associates a collection of filtrations to $\sF$ as follows. We first set
\[
	V_\sF=\H^0\left(T,\sF_{|T}\right)^T; 
\]
that is, $V_\sF$ is the $\k$-vector space obtained as the $T$-invariant sections of the restriction of $\sF$ to the torus $T$. The restriction of $\sF$ to $T$ is a vector bundle, and $V_\sF$ may be identified with the fibre of this bundle over the identity element of $T$. In particular, it is a vector space of dimension equal to the rank of $\sF$. 

For each ray $\rho\in\Sigma^{(1)}$, we may consider the decreasing $\ZZ$-filtration $V_\sF^\rho$ defined as
\[
	V_\sF^\rho(i)=\{z\in V_\sF\ |\ \ord_{D_\rho}(z) \geq i\}.
\]
As before $\ord_{D_\rho}(z)$ denotes the order of vanishing of a section $z$ along the prime divisor $D_\rho$.
When the sheaf $\sF$ is clear from the context, we will omit the subscript and use the notation $V$ and $V^\rho(i)$.

\begin{example}[The reflexive hull of the cotangent sheaf]\label{ex:cotangent}
	Let $X_\Sigma$ be a toric variety with cotangent sheaf $\Omega=\Omega_{X_\Sigma}$. This bundle has a natural $T$-equivariant structure. The corresponding filtrations for its reflexive hull $\hat\Omega$ are as follows:
	\begin{align*}
V&=M\otimes \k,\\
V^\rho(i)&=\begin{cases}
V & i<0,\\
\ker\left(\nu_\rho\right) \subset V & i=0,\\
0 & i > 0.
\end{cases}
	\end{align*}
	If $X_\Sigma$ is smooth, then $\hat\Omega=\Omega$ and this is just \cite[Section~2.3, Example 5]{Klyachko1989}. For the singular case we note that $\hat\Omega$ agrees with $\Omega$ on the non-singular locus of $X_\Sigma$. Since any toric variety is smooth in codimension~$1$, the filtrations for $\hat\Omega$ agree with the filtrations for the restriction of $\Omega$ to the non-singular locus of $X_\Sigma$, which are exactly the filtrations above.
\end{example}

Let $\sF$ be an equivariant reflexive sheaf on $X_\Sigma$. It is straightforward to describe the filtration data of the reflexive hull of its symmetric powers $\hat S^m\sF$ in terms of the filtration data of $\sF$:
\begin{align*}
	&V_{\hat S^m\sF}=S^m V_{\sF},\\
	&V_{\hat S^m\sF}^\rho(i)=\sum_{j_1+j_2+\ldots+j_m=i}V_{\sF}^\rho(j_1)\cdot \ldots \cdot V_{\sF}^\rho(j_m)\subseteq S^m V_\sF.
\end{align*}
See~\cite[Corollary 3.5]{jose} for the locally free case; the reflexive case follows immediately.

\begin{example}[Symmetric powers of the cotangent sheaf]\label{ex:Scotangent}
	Combining the above with Example \ref{ex:cotangent}, we obtain that for the reflexive sheaf $\hat S^m \Omega$, we have
\begin{align*}
	V_{\hat S^m\Omega}&=S^m (M\otimes \k),\\
		V_{\hat S^m \Omega}^\rho(i)&=\begin{cases}
S^m (M\otimes \k) & i\leq -m,\\
S^{i+m}(\rho^\perp)\cdot S^{-i} (M\otimes \k)  & -m\leq i \leq 0,\\
0 & i \geq  1.
\end{cases}
\end{align*}
\end{example}

For any $T$-equivariant reflexive sheaf $\sF$, $T$ acts on the cohomology groups $\H^p(X_\Sigma,\sF)$, and so these decompose as a direct sum of eigenspaces
\[
	\H^p(X_\Sigma,\sF)=\bigoplus_{u\in M} \H^p(X_\Sigma,\sF)_u.
\]
Global sections are especially easy to describe. For any ray $\rho\in\Sigma$ and $u\in M$, let $\rho(u)=\langle \nu_\rho,u\rangle$.
We have that $\H^0(T,\sF|_T)_u\simeq V_{\sF}$ via $z \mapsto x^{u}z$. 
From $\ord_{D_\rho}(x^{u} z)=\ord_{D_\rho}(z)+\rho(u)$ we obtain
\begin{equation}\label{E:shifted_filtration}
\left\{z \in \H^0(T,\sF|_T)_u \mid \ord_{D_\rho}(z)\geq i \right\}=
x^{-u} V_{\sF}^\rho(i+\rho(u)).
\end{equation}
In particular,
\begin{equation}\label{eqn:global}
	\H^0(X_\Sigma,\sF)_u\cong\bigcap_{\rho\in\Sigma^{(1)}} V_\sF^\rho(\rho(u)).
\end{equation}

Higher cohomology groups of $\sF$ may also be recovered from the filtration data. For $\sigma\in\Sigma$ and $u\in M$, set
\[
	W_\sF^\sigma(u)=V_\sF/\sum_{\rho\in\Sigma^{(1)}\cap \sigma} V_\sF^\rho(\rho(u)).
\]
Klyachko uses these vector spaces to construct a complex
\begin{equation}\label{eqn:cohom}
	0 \lra \bigoplus_{\sigma\in\Sigma^{(0)}} W^\sigma(u) \lra \bigoplus_{\sigma\in\Sigma^{(1)}} W^\sigma(u) \lra \bigoplus_{\sigma\in\Sigma^{(2)}} W^\sigma(u) \lra \cdots
\end{equation}
whose $p$\textsuperscript{th} cohomology may be identified with $\H^p(X_\Sigma,\sF)_u$; see \cite[Theorem 4.1.1]{Klyachko1989}. In particular, we have the following.

\begin{proposition}\label{prop:chiu}
Let $\sF$ be a $T$-equivariant reflexive sheaf on the toric variety $X_\Sigma$. For any $u\in M$ the quantity
\[
	\chi_u(\sF):=\sum_{p\geq 0} (-1)^p \dim \H^p(X_\Sigma,\sF)_u
\]
may be computed as
\[
	\chi_u(\sF)=\sum_{p\geq 0} (-1)^p \sum_{\sigma\in\Sigma^{(p)}} \dim W_\sF^\sigma(u).
\]
\end{proposition}
\begin{proof}
	Since the cohomology of the complex \eqref{eqn:cohom} computes $\H^p(X_\Sigma,\sF)_u$, the alternating sum of the dimensions of the terms of the complex computes $\chi_u(\sF)$.
\end{proof}

\begin{remark}
	Klyachko initially constructs the complex \eqref{eqn:cohom} when $\sF$ is locally free. However, it is straightforward to check that the result \cite[Theorem 4.1.1]{Klyachko1989} is also true in the reflexive case; the proof in \textit{loc.~cit.}~goes through verbatim.
\end{remark}

\subsection{Ehrhart theory}
We briefly recall some basics of Ehrhart theory. See \textit{e.g.}~\cite{polytopes} for details.
For the purposes of this article, a \emph{convex polytope} is the convex hull of a finite set in $\RR^d$. A \emph{non-convex polytope} is a connected finite union of convex polytopes. A \emph{half-open polytope} is a polytope with some of its faces removed.
A \emph{quasi-polynomial} $f(t)$ is a function from $\NN$ to $\NN$ that may be written in the form
\[
	f(t)=a_d(t)t^d+a_{d-1}(t)t^{d-1}+\dots+a_{0}(t),
        \]
where the coefficient functions $a_i(t)$ are periodic of integral period. 
The degree of such an $f(t)$ is the largest exponent $d$ such that $a_d(t)$ is not identically zero; the period is the least common multiple of the periods of all coefficient functions.

For a rational polytope $\sP\subset \RR^d$, we may consider its Ehrhart function
\[L(\sP,t)=\#\left(t\sP \cap \ZZ^d\right)\quad\text{for }t=0,1,2,\ldots .\]
This function is a quasi-polynomial in $t$ whose period divides the smallest integer $\lambda$ such that $\lambda\cdot \sP$ is integral. The degree of 
$L(\sP,t)$ is the dimension of $\sP$. Assuming that $\sP$ has dimension $d$, the leading coefficient of $L(\sP,t)$ is simply the volume of $\sP$.

Given a subset $A\subset \RR^d$, we define its \emph{lattice point transform} to be
\[
	\lpt_A=\sum_{u\in A\cap\ZZ^d} z^u.
\]
This is a formal power series in $z_1,\ldots,z_d$ and is a useful tool for computing the generating series of $L(\sP,t)$. We will make use of the following. 

\begin{proposition}[\textit{cf.}~{\cite[Theorem 3.5]{polytopes}}]\label{prop:lpt}
Let $C\subset \RR^d$ be a simplicial cone whose rays are generated by primitive vectors $w_1,\ldots,w_k\in \ZZ^d$. Set
\[
	\Pi(C)=\left\{\sum \alpha_i w_i\ |\ 0\leq \alpha_i <1\right\}.
\]
Then\[	
\lpt_C(z)=\frac{\lpt_{\Pi(C)}}{(1-z^{w_1})\cdots(1-z^{w_k})}.
\]
\end{proposition}

\section{Computation of \texorpdfstring{$\boldsymbol{\chiloc}$}{chi\_loc}}
\subsection{A recursive formula}
Let $Y\to X$ be a minimal resolution of a surface $X$ with an $A_n$-singularity $s_n$. We 
are interested in computing 
\[
\chi(n,m):=\chiloc\left(s_n,S^m\Omega^1_Y\right).
\]
Our approach is to use the machinery described in Section~\ref{sec:klyachko}. It will be advantageous to first develop a recursive formula for $\chi(n,m)$. For $n=0$ we set $\chi(n,m)=\chi(0,m)=0$.

Fix $N=\ZZ^2$. As in Example \ref{ex:Anres} we let $\rho_i\subset \RR^2$ be the ray generated by $(i,1)$. We additionally consider the rays $\rho_+,\rho_-,\rho_\infty$ generated by $(1,0)$, $(-1,0)$, and $(0,-1)$, respectively. Fixing $n\geq 1$, we let $\widetilde\Sigma$, $\overline\Sigma$, and $\Sigma$ be the unique complete fans in $\RR^2$ whose rays are as follows:
\begin{align*}
	\widetilde\Sigma^{(1)}&=\{\rho_0,\ldots,\rho_{n+1},\rho_+,\rho_\infty,\rho_-\},\\
	\overline\Sigma^{(1)}&=\{\rho_0,\rho_1,\rho_{n+1},\rho_+,\rho_\infty,\rho_-\},\\
	\Sigma^{(1)}&=\{\rho_0,\rho_{n+1},\rho_+,\rho_\infty,\rho_-\}.
\end{align*}
See Figure \ref{fig:fans}.

For any $m\geq 0$ and $u\in M=\ZZ^2$, we define
\[
\delta_n(m,u):=\chi_u\left(\hat S^m \Omega_{X_{\Sigma}}\right)-\chi_u\left(\hat S^m \Omega_{X_{\overline\Sigma}}\right).
\]
We will see in Section~\ref{sec:filtrations} how to calculate $\delta_n(m,u)$ explicitly.
We set 
\[
	\delta_n(m)=\sum_{u\in M} \delta_n(m,u).
\]
Since both $X_{\overline\Sigma}$ and $X_{\Sigma}$ are complete, $\delta_n(m,u)=0$ for all but finitely many $u\in M$, and the above sum is finite.
\begin{lemma}\label{lemma:induct}
For any $m\geq 0$,
\[
	\chi(n,m)-\chi(n-1,m)=\delta_n(m).
\]
\end{lemma}

\begin{proof}
The toric varieties $X_{\widetilde\Sigma}$, $X_{\overline\Sigma}$, and $X_{\Sigma}$ are all complete surfaces. Similarly to Example \ref{ex:Anres}, there is a sequence of toric morphisms 
\[X_{\widetilde\Sigma}\longrightarrow X_{\overline\Sigma}\longrightarrow X_{\Sigma}.\]
The surface $X_\Sigma$ has a single $A_n$-singularity (see Example \ref{ex:An}). The surface $X_{\overline\Sigma}$ has a single $A_{n-1}$-singularity: this may be seen by applying the lattice isomorphism 
\[
\left(	\begin{array}{c c}
	1&-1\\
	0&1\\
\end{array}\right)\in\SL(2,\ZZ).
\]
Here, an $A_0$-singularity is just a smooth point. As in Example \ref{ex:Anres}, $X_{\widetilde\Sigma}$ is the minimal resolution of both $X_\Sigma$ and $X_{\overline\Sigma}$.

By applying \eqref{eqn:blache} for both the $A_n$- and the $A_{n-1}$-singularity, we obtain
\begin{align*}
	\chi(n,m)-\chi(n-1,m)=\left(\chi\left(\hat S^m \Omega_{X_{\Sigma}}\right)-\chi\left(S^m \Omega_{X_{\widetilde\Sigma}}\right)\right)-\left(\chi\left(\hat S^m \Omega_{X_{\overline \Sigma}}\right)-\chi\left(S^m \Omega_{X_{\widetilde\Sigma}}\right)\right)\\
	=\chi\left(\hat S^m \Omega_{X_{\Sigma}}\right)-\chi\left(\hat S^m \Omega_{X_{\overline\Sigma}}\right), 
\end{align*}
and the claim follows.
\end{proof}

\begin{figure}
	\begin{tabular}{c c c}
		\scriptsize{
			\begin{tikzpicture}[scale=.3]
		\draw[->] (0,0) -- (7.8,1.3);
		\draw[->] (0,0) -- (0,1.3);
		\draw[->] (0,0) -- (5.2,1.3);
		\draw[->] (0,0) -- (3.9,1.3);
		\draw[->] (0,0) -- (2.6,1.3);
		\draw[->] (0,0) -- (1.3,1.3);
		\draw[->] (0,0) -- (6.5,1.3);
		\draw[->] (0,0) -- (2,0);
		\draw[->] (0,0) -- (-2,0);
		\draw[->] (0,0) -- (0,-1.3);
\node[right] at  (2,0) {$\rho_+$};
\node[left] at  (-2,0) {$\rho_-$};
\node[below] at  (0,-1.3) {$\rho_\infty$};
\node[above right] at  (0,1.3) {$\rho_0$};
\node[above right] at  (1.3,1.3) {$\rho_1$};
\node[above right] at  (2.6,1.3) {$\rho_2$};
\node[above right] at  (6.5,1.3) {$\rho_n$};
\node[above right] at  (7.8,1.3) {$\rho_{n+1}$};
\node[above ] at  (5.3,1.3) {$\cdots\cdots$};
	\end{tikzpicture} 
}		&
{\tiny			\begin{tikzpicture}[scale=.3]
		\draw[->] (0,0) -- (7.8,1.3);
		\draw[->] (0,0) -- (0,1.3);
		\draw[->] (0,0) -- (1.3,1.3);
		\draw[->] (0,0) -- (2,0);
		\draw[->] (0,0) -- (-2,0);
		\draw[->] (0,0) -- (0,-1.3);
\node[right] at  (2,0) {$\rho_+$};
\node[left] at  (-2,0) {$\rho_-$};
\node[below] at  (0,-1.3) {$\rho_\infty$};
\node[above right] at  (0,1.3) {$\rho_0$};
\node[above right] at  (1.3,1.3) {$\rho_1$};
\node[above right] at  (7.8,1.3) {$\rho_{n+1}$};
	\end{tikzpicture} 
}&
{\tiny			\begin{tikzpicture}[scale=.3]
		\draw[->] (0,0) -- (7.8,1.3);
		\draw[->] (0,0) -- (0,1.3);
		\draw[->] (0,0) -- (2,0);
		\draw[->] (0,0) -- (-2,0);
		\draw[->] (0,0) -- (0,-1.3);
\node[right] at  (2,0) {$\rho_+$};
\node[left] at  (-2,0) {$\rho_-$};
\node[below] at  (0,-1.3) {$\rho_\infty$};
\node[above right] at  (0,1.3) {$\rho_0$};
\node[above right] at  (7.8,1.3) {$\rho_{n+1}$};
	\end{tikzpicture} 
}\\
$\widetilde{\Sigma}$&
$\overline{\Sigma}$
&$\Sigma$
\end{tabular}	
\caption{The fans $\widetilde\Sigma$, $\overline\Sigma$, and $\Sigma$}\label{fig:fans}
\end{figure}

\subsection{Computing $\boldsymbol{\delta_n(m)}$}\label{sec:filtrations}
Define
\[
	\lambda_m(i)=\begin{cases}
0 & i\leq -m,\\
i+m &-m \leq i \leq 1,\\
m+1 & i \geq  1.
	\end{cases}
\]
\begin{lemma}\label{lemma:delta}For any $u\in M=\ZZ^2$, 
\begin{align*}
\delta_n(m,u)=(m+1)-\lambda_m(\rho_1(u))&-\max \{m+1-\lambda_m(\rho_0(u))-\lambda_m(\rho_1(u)),0\}\\
&-\max \{m+1-\lambda_m(\rho_1(u))-\lambda_m(\rho_{n+1}(u)),0\}\\
&+\max \{m+1-\lambda_m(\rho_0(u))-\lambda_m(\rho_{n+1}(u)),0\}.
\end{align*}
\end{lemma}
\begin{proof}
	We let $V$ and $\{V^\rho(i)\}$ be the vector space and filtrations associated to the reflexive hull of the $\supth{m}$ symmetric power of the cotangent sheaf on any toric surface. Then by Example \ref{ex:Scotangent} we have
\begin{align*}
\dim V&=m+1,\\
	\dim V^\rho(i)&=m+1-\lambda_m(i),\\
	\dim V^\rho(i)\cap V^{\rho'}{(j)}&=\max \{m+1-\lambda_m(i)-\lambda_m(j),0\}\quad\mathrm{if}\ \rho\neq \rho'.
\end{align*}

For $0\leq i,j\leq n+1$ let $\sigma_{ij}$ denote the cone in $\RR^2$ spanned by $\rho_i$ and $\rho_j$. We have that $\Sigma^{(0)}=\overline\Sigma^{(0)}$, and the rays of $\Sigma$ and $\overline \Sigma$ differ only by $\rho_1$ (which belongs to $\overline \Sigma$). The sets $\Sigma^{(2)}$ and $\overline\Sigma^{(2)}$ differ only by $\sigma_{01},\sigma_{1(n+1)}$, which belong to $\overline\Sigma^{(2)}$,  and $\sigma_{0(n+1)}$, which belongs to $\Sigma^{(2)}$. 
Applying Proposition \ref{prop:chiu} to both 
$\chi_u(S^m \hat \Omega_{X_{\overline\Sigma}})$ and $\chi_u(S^m \hat \Omega_{X_{\Sigma}})$ and cancelling terms, we obtain
\begin{align*}
\delta_n(m,u)&=\dim W^{\rho_1}(u)-\dim W^{\sigma_{01}}(u)-\dim W^{\sigma_{1(n+1)}}(u)+\dim W^{\sigma_{0(n+1)}}(u)\\
&=\begin{aligned}[t]-\dim V^{\rho_1}(\rho_1(u))&+\dim (V^{\rho_0}(\rho_0(u))+V^{\rho_1}(\rho_1(u)))\\
&+\dim (V^{\rho_1}(\rho_1(u))+V^{\rho_{n+1}}(\rho_{n+1}(u)))\\
&-\dim (V^{\rho_0}(\rho_0(u))+V^{\rho_{n+1}}(\rho_{n+1}(u)))
\end{aligned}\\
&=\begin{aligned}[t](m+1)-\lambda_m(\rho_1(u))&-\max \{m+1-\lambda_m(\rho_0(u))-\lambda_m(\rho_1(u)),0\}\\
&-\max \{m+1-\lambda_m(\rho_1(u))-\lambda_m(\rho_{n+1}(u)),0\}\\
&+\max \{m+1-\lambda_m(\rho_0(u))-\lambda_m(\rho_{n+1}(u)),0\}.
\end{aligned}
\end{align*}
The second equality follows by writing $W^\sigma$ in terms of $V^\rho$. The third follows by using
\[
	\dim (V^\rho(i)+V^{\rho'}(j))=\dim V^\rho(i)+\dim V^{\rho'}(j)-\dim (V^\rho(i)\cap V^{\rho'}(j))
\]
and the above computation of $\dim (V^\rho(i)\cap V^{\rho'}(j))$.
\end{proof}

Using the formula for $\delta_n(m,u)$ in Lemma \ref{lemma:delta}, we may extend $\delta_n(m,u)$ to a function in $u$  on all of $\RR^2$; this function is piecewise linear. 
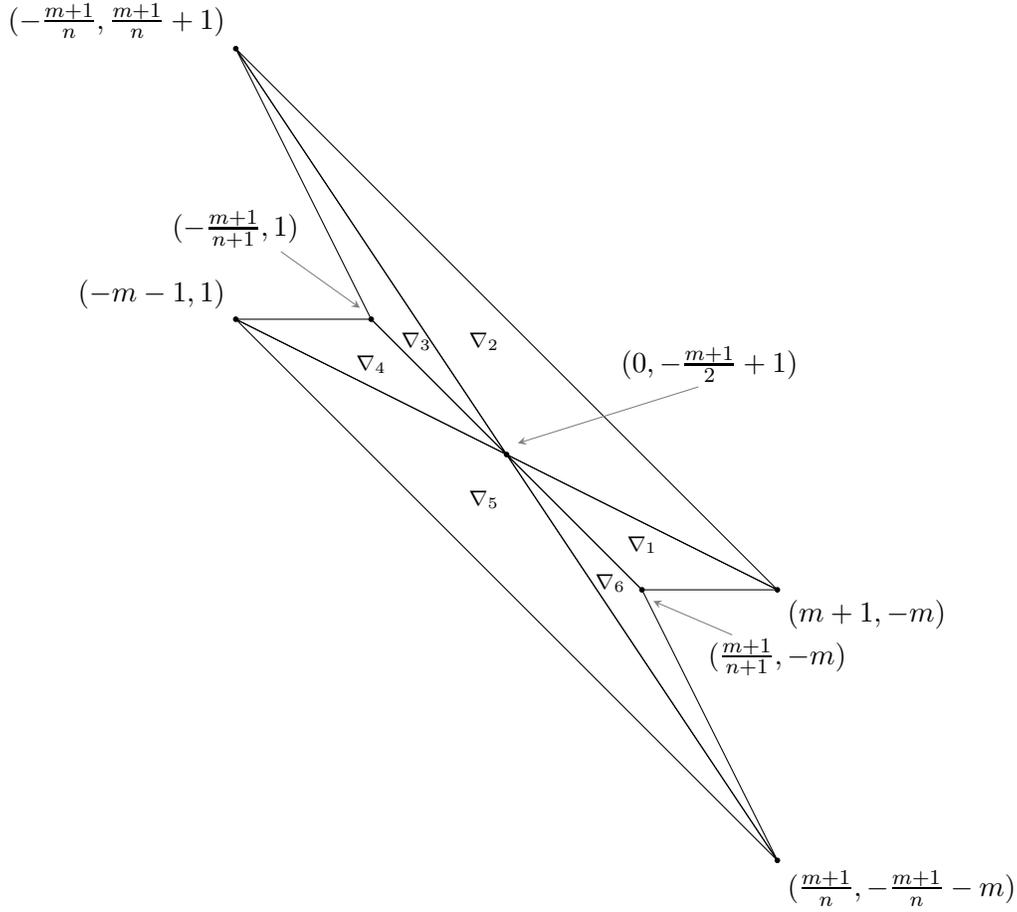
\begin{figure}
\centering
	\begin{tikzpicture}[scale=.3]
\draw (-12,0) -- (-6,0) -- (0,-6) -- (-12,0);
\draw (-12,0) -- (12,-24) -- (0,-6) -- (-12,0);
\draw (-6,0) -- (0,-6) -- (-12,12) -- (-6,0);
\draw (-12,12) -- (0,-6) -- (12,-12) -- (-12,12);
\draw (12,-12) -- (0,-6) -- (6,-12) -- (12,-12);
\draw (6,-12) -- (0,-6) -- (12,-24) -- (6,-12);
\filldraw (-12,0) circle [radius=.1];
\filldraw (12,-24) circle [radius=.1];
\filldraw (12,-12) circle [radius=.1];
\filldraw (-12,12) circle [radius=.1];
\filldraw (-6,0) circle [radius=.1];
\filldraw (6,-12) circle [radius=.1];
\filldraw (0,-6) circle [radius=.1];
\node[above left] at  (-12,0) {$(-m-1,1)$};
\node[below right] at  (12,-24) {$\left(\frac{m+1}{n},-\frac{m+1}{n}-m\right)$};
\node[above left] at  (-12,12) {$\left(-\frac{m+1}{n},\frac{m+1}{n}+1\right)$};
\node[below right] at  (12,-12) {$(m+1,-m)$};
\node at  (12,-15) {$\left(\frac{m+1}{n+1},-m\right)$};
\node at  (-12,4) {$\left(-\frac{m+1}{n+1},1\right)$};
\draw[-stealth,gray] (-10,3) -- (-6.5,.5);
\draw[-stealth,gray] (10,-14) -- (6.5,-12.5);
\node at  (9,-2) {$\left(0,-\frac{m+1}{2}+1\right)$};
\draw[-stealth,gray] (8.5,-3) -- (.5,-5.5);
\scriptsize{
	\node at (-1,-1) {$\nabla_2$};
	\node at (-4,-1) {$\nabla_3$};
	\node at (-6,-2) {$\nabla_4$};
	\node at (-1,-8) {$\nabla_5$};
	\node at (4.6,-11.7) {$\nabla_6$};
	\node at (6,-10) {$\nabla_1$};
}

\end{tikzpicture}
	\caption{Regions of linearity of $\delta_n(m,u)$}\label{fig:linearity}
\end{figure}
\begin{lemma}\label{lemma:linearity}
	Outside of the six polytopes $\nabla_1,\ldots,\nabla_6$ pictured in Figure \ref{fig:linearity}, the function $\delta_n(m,u)$ vanishes. 
	The regions of linearity of $\delta_n(m,u)$ are exactly the six polytopes $\nabla_1,\ldots,\nabla_6$. On each of these six simplices, $\delta_n(m,u)$ takes value 
$(m+1)/2$ at the vertex $(0,-(m+1)/2+1)$ and $0$ at the other two vertices.	
\end{lemma}
\begin{proof}
	From the description of $\delta_n(m,u)$ in Lemma \ref{lemma:delta} and the definition of $\lambda_m(i)$, it follows that the non-linear locus of $\delta_n(m,u)$ is contained in the lines $\rho_i(u)=1$, $\rho_j(u)=-m$ for $i,j=0,1,n+1$ along with the lines $\rho_0(u)+\rho_1(u)=1-m$, $\rho_1(u)+\rho_{n+1}(u)=1-m$, and $\rho_0(u)+\rho_{n+1}(u)=1-m$. 

	Since $\delta_n(m,u)=0$ for all but finitely many $u\in \ZZ^2$, we know that $\delta_n(m,u)=0$ on any unbounded region in the above subdivision of $\RR^2$. For each of the remaining bounded regions, we may calculate the linear function representing $\delta_n(m,u)$ on that region. In doing so, and combining regions with the same linear function, one obtains the result of the lemma.
\end{proof}

\subsection{Counting lattice points}
For this subsection we introduce some notation for subsets of $\RR^2$. Let $\gamma=(0,1/2)\in\RR^2$.
For $a\leq b\in \QQ$ set 
\[
	\lc a:b \rc=\Conv\{(a,0),(b,0),(0,1/2)\}\setminus \gamma.
\]
We further define
\[
\lc a \rc=\lc a:a\rc,\quad \lo a:b\rc=\lc a:b \rc \setminus \lc a\rc, \quad \lc a: b\ro=\lc a:b\rc \setminus \lc b\rc, \quad \lo a:b\ro =\lc a:b\rc \setminus (\lc a\rc \cup \lc b\rc ).
\]
For sets $A,B\subset \RR^2$ we will use the notation $A\bm{+}B$ to denote a disjoint union of $A$ and $B$ as abstract sets. Likewise, for $\ell\in\ZZ$ we use $\ell \bm{*} A$ to denote the disjoint union of $A$ with itself $\ell$ times (again as an abstract set).
In particular, $\#((\ell\bm{*} A)\cap\ZZ^2)=\ell\cdot (\#(A\cap \ZZ^2))$.

We set
\begin{align*}
	\Box_n:=&2 \bm{*} \left(-\frac{1}{n}:-\frac{1}{n+1}\right)\bm{+}2\bm{*}\left(-1:\frac{1}{n}\right)\bm{+}2\bm{*}\left(\frac{1}{n+1}:1\right)\\
	&\bm{+}2\bm{*}\left[\frac{1}{n+1}\right]\bm{+}2\bm{*}\left[\frac{1}{n}\right]\bm{+}2\bm{*}\lc1\rc\bm{+}\gamma.
\end{align*}
By $(m+1)\cdot \Box_n$ we denote the $\supst{(m+1)}$ dilate of $\Box_n$, where the dilate of a disjoint union is the disjoint union of the dilates.
\begin{lemma}\label{lemma:Q}
For any $m\geq 1$,
	\[
	\delta_n(m)=\sum_{(x,y)\in ((m+1)\cdot \Box_n)\cap\ZZ^2} y.
\]
\end{lemma}
\begin{proof}
	To each polytope $\nabla_i$ from Figure \ref{fig:linearity}, we will apply an invertible integral affine linear transformation $\phi_i$:
\[
	\def\arraystretch{1.2}
	\begin{array}{c c c}
		\textrm{Polytope} & \textrm{Transformation }\phi_i & \textrm{Image}
		\\
		\hline
		\nabla_1 & (x,y)\mapsto (x,y+m) & (m+1)\cdot \lc \frac{1}{n+1}:1\rc\\
		\nabla_2 & (x,y)\mapsto (-x,-x-y+1) & (m+1)\cdot \lc -1:\frac{1}{n}\rc\\
		\nabla_3 & (x,y)\mapsto (x,(n+1)x+y+m) & (m+1)\cdot \lc \frac{-1}{n}:\frac{-1}{n+1}\rc\\
		\nabla_4 & (x,y)\mapsto (-x,-y+1) & (m+1)\cdot \lc \frac{1}{n+1}:1\rc\\
		\nabla_5 & (x,y)\mapsto (x,x+y+m) & (m+1)\cdot \lc -1:\frac{1}{n}\rc\\
		\nabla_6 & (x,y)\mapsto (-x,-(n+1)x-y+1) & (m+1)\cdot \lc \frac{-1}{n}:\frac{-1}{n+1}\rc\\
	\end{array}
\]
Note that the transformations $\phi_i$ and $\phi_{i+1}$ agree along $\nabla_i\cap \nabla_{i+1}$, with indices taken modulo $6$. 
It follows from Lemma \ref{lemma:linearity} that for each $i$ and each $(x,y)\in \phi_i(\nabla_i)$, we have
\[
	\delta_n\left(m,\phi_i^{-1}((x,y))\right)=y.
\]
Again using Lemma \ref{lemma:linearity}, we have
\begin{align*}
	\delta_n(m)=\sum_{u\in (\bigcup \nabla_i)\cap M} \delta_n(m,u).
\end{align*}
Applying $\phi_i$ to each $\nabla_i$ and using inclusion-exclusion, we obtain
the claim of the lemma.
\end{proof}

We are now able to use induction to obtain a formula for $\chi(n,m)$ as a weighted lattice point count.
Using notation introduced at the start of this subsection, define
\[
	\Delta_n=2\bm{*}\left(\frac{1}{n+1}:2(n+1)-\frac{1}{n+1}\right]\bm{+}n\bm{*}\gamma.
\]
\begin{theorem}\label{thm:delta}
For $n,m\geq 1$ we have
\[
	\chi(n,m)=\sum_{(x,y)\in ((m+1)\cdot \Delta_n)\cap \ZZ^2} y.
\]
\end{theorem}
\begin{proof}
Up to integral translation in the $x$-direction and reflection around the line $x=0$, we have
\[\Box_n\equiv 2\bm{*}\left(\left(-\frac{1}{n}:-\frac{1}{n+1}\right]\bm{+}\left(\frac{1}{n+1}:2+\frac{1}{n}\right]\right)\bm{+}\gamma.
\]
Indeed, 
\[\left(-\frac{1}{n}:-\frac{1}{n+1}\right)\bm{+}\left[\frac{1}{n+1}\right]\equiv\left(-\frac{1}{n}:-\frac{1}{n+1}\right)\bm{+}\left[\frac{-1}{n+1}\right]\equiv \left(-\frac{1}{n}:-\frac{1}{n+1}\right]\]
and 
\begin{align*}
	\left(-1:\frac{1}{n}\right)\bm{+}\left(\frac{1}{n+1}:1\right)\bm{+}\left[ \frac{1}{n}\right]\bm{+}\left[ 1\right]
	&\equiv \left( 1:2+\frac{1}{n}\right)\bm{+}\left(\frac{1}{n+1}:1\right)\bm{+}\left[ 2+\frac{1}{n}\right]\bm{+}\left[ 1\right]\\
	&\equiv \left(\frac{1}{n+1}:2+\frac{1}{n}\right].
\end{align*}

By translating the set $\lo\frac{1}{k+1}:2+\frac{1}{k}\rc$ in $\Box_k$ by $2(n-k)$ to the right, it is straightforward to see that
\begin{align*}
	\Box_1\bm{+}\cdots\bm{+}\Box_n\equiv2\bm{*}\left( \left(-1:-\frac{1}{n+1}\right]\bm{+}\left(\frac{1}{n+1},2n+1\right]\right)\bm{+}n\bm{*}\gamma \equiv \Delta_n.
\end{align*}
Since $\chi(n,m)=\delta_1(m)+\dots+\delta_n(m)$ by Lemma \ref{lemma:induct}, the claim of the theorem follows from Lemma \ref{lemma:Q}.
\end{proof}

\subsection{Generating functions}\label{S:generating_functions}
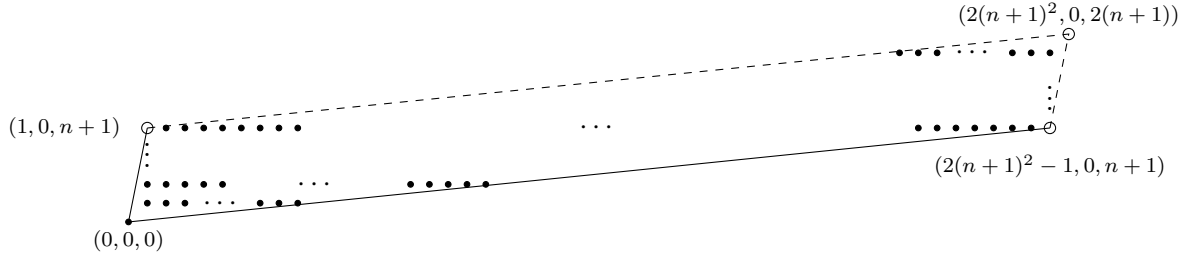
\begin{figure}
	\centering
	\begin{tikzpicture}[scale=.25]
\draw (0,0) --(1,5);
\draw (0,0) --(49,5);
\draw[dashed] (1,5) -- (50,10) -- (49,5);
\filldraw (0,0) circle [radius=.15];
\filldraw (1,1) circle [radius=.15];
\filldraw (2,1) circle [radius=.15];
\filldraw (3,1) circle [radius=.15];
\filldraw (7,1) circle [radius=.15];
\filldraw (8,1) circle [radius=.15];
\filldraw (9,1) circle [radius=.15];
\filldraw (1,2) circle [radius=.15];
\filldraw (2,2) circle [radius=.15];
\filldraw (3,2) circle [radius=.15];
\filldraw (4,2) circle [radius=.15];
\filldraw (5,2) circle [radius=.15];

\filldraw (15,2) circle [radius=.15];
\filldraw (16,2) circle [radius=.15];
\filldraw (17,2) circle [radius=.15];
\filldraw (18,2) circle [radius=.15];
\filldraw (19,2) circle [radius=.15];

\filldraw (2,5) circle [radius=.15];
\filldraw (3,5) circle [radius=.15];
\filldraw (4,5) circle [radius=.15];
\filldraw (5,5) circle [radius=.15];
\filldraw (6,5) circle [radius=.15];
\filldraw (7,5) circle [radius=.15];
\filldraw (8,5) circle [radius=.15];
\filldraw (9,5) circle [radius=.15];

\filldraw (48,5) circle [radius=.15];
\filldraw (47,5) circle [radius=.15];
\filldraw (46,5) circle [radius=.15];
\filldraw (45,5) circle [radius=.15];
\filldraw (44,5) circle [radius=.15];
\filldraw (43,5) circle [radius=.15];
\filldraw (42,5) circle [radius=.15];
\draw (49,5) circle [radius=.3];
\draw (1,5) circle [radius=.3];
\draw (50,10) circle [radius=.3];

\node at (5,1) {$\cdots$};
\node at (10,2) {$\cdots$};
\node at (25,5) {$\cdots$};
\node at (45,9) {$\cdots$};
\node at (1,4) {$\vdots$};
\node at (49,7) {$\vdots$};
\scriptsize{
\node [below] at (0,0) {$(0,0,0)$};
\node [left] at (0,5) {$(1,0,n+1)$};
\node [above] at (50,10) {$\left(2(n+1)^2,0,2(n+1)\right)$};
\node [below] at (49,4) {$\left(2(n+1)^2-1,0,n+1\right)$};
}

\filldraw (41,9) circle [radius=.15];
\filldraw (42,9) circle [radius=.15];
\filldraw (43,9) circle [radius=.15];
\filldraw (47,9) circle [radius=.15];
\filldraw (48,9) circle [radius=.15];
\filldraw (49,9) circle [radius=.15];

\end{tikzpicture}
	\caption{Lattice points of the region $\Pi(C_1)$}\label{fig:pi}
\end{figure}
\begin{lemma}\label{lemma:gen}
The regular generating function for $\chi(n,m)$ as a function of $m$ is
\[
	\sum_{m\geq 0} \chi(n,m)z^m=
\frac{z\cdot
	\left( (n+1)(1+z+\dots+z^{n})^2-(1+z^2+\dots+z^{2n}) \right)
}{(1-z)^2(1-z^{n+1})^2}. 
\]
\end{lemma}
\begin{proof}
Consider the cones
\begin{align*}
	C_1&=\Pos\left\{(1,0,n+1),(2(n+1)^2-1,0,n+1),(0,1,2)\right\},\\
	C_2&=\Pos\{(1,0,n+1),(0,1,2)\},\\
	C_3&=\Pos\{(0,1,2)\},
\end{align*}
where $\Pos$ denotes the positive hull. These are the cones in $\RR^3$ over $\lc \frac{1}{n+1}:2(n+1)-\frac{1}{n+1}\rc$, $\lc \frac{1}{n+1}\rc $, and $\gamma$.

Using variables $x,y,z$ and following notation from Proposition \ref{prop:lpt}, we have
\begin{equation}\label{eqn:lpt}
	\lpt_{\Pi(C_1)}=1+\left(\sum_{k=1}^{n} \sum_{j=1}^{2(n+1)k-1}(x^jz^k+x^{2(n+1)^2-j}z^{2(n+1)-k}) \right)+\sum_{j=2}^{2(n+1)^2-2}x^jz^{n+1}.
\end{equation}
Indeed, fixing the third coordinate equal to some integer $k$ with $1\leq k \leq n$, $\Pi(C_1)$ contains lattice points with first coordinate ranging from 
\[\left\lceil \frac{k}{n+1} \right\rceil=1\qquad \textrm{to}\qquad \left\lfloor 2(n+1)k-\frac{k}{n+1}\right\rfloor=2(n+1)k-1.\] For third coordinate equal to $n+1$, we have lattice points with first coordinate ranging from $2$ up to $2(n+1)^2-2$. Points with third coordinate larger than $n+1$ are obtained by reflecting points with $1\leq k \leq n$ through the point $(n+1,0,(n+1)^2)$. 
See Figure \ref{fig:pi}.

Further, we note that $\lpt_{\Pi(C_2)}=\lpt_{\Pi(C_3)}=1$. By Proposition \ref{prop:lpt} we conclude that
\begin{align*}
	\lpt_{C_1}&=\frac{\lpt_{\Pi(C_1)}}{(1-xz^{n+1})\left(1-x^{2(n+1)^2-1}z^{n+1}\right)}\cdot \frac{1}{(1-yz^2)},\\
	\lpt_{C_2}&=\frac{1}{(1-xz^{n+1})}\cdot\frac{1}{(1-yz^2)},\\
	\lpt_{C_3}&=\frac{1}{(1-yz^2)}.
\end{align*}

By the definition of $\Delta_n$, 
$\#((m+1)\cdot \Delta_n)\cap \ZZ^2$ is the coefficient of $z^{m+1}$ in
\[2\cdot \lpt_{C_1}(1,1,z)-2\cdot \lpt_{C_2}(1,1,z)+n\cdot \lpt_{C_3}(1,1,z).\]
Similarly, 
\[\sum_{(x,y)\in((m+1)\cdot \Delta_n)\cap \ZZ^2} y\]
is the coefficient of $z^{m+1}$ in 
\[2\cdot \frac{\partial \lpt_{C_1}}{\partial y}(1,1,z)-2\cdot \frac{\partial \lpt_{C_2}}{\partial y}(1,1,z)+n\cdot \frac{\partial \lpt_{C_3}}{\partial y}(1,1,z).\]
Applying 
Theorem \ref{thm:delta} and using the above expressions for the lattice point transforms, we obtain
\begin{align*}
\sum_{m\geq 0} \chi(n,m)z^m&=
\frac{1}{z}\left(2\cdot \frac{\partial \lpt_{C_1}}{\partial y}(1,1,z)-2\cdot \frac{\partial \lpt_{C_2}}{\partial y}(1,1,z)+n\cdot \frac{\partial \lpt_{C_3}}{\partial y}(1,1,z)\right)\\
&=\frac{z}{(1-z^2)^2}\cdot \frac{2\lpt_{\Pi(C_1)}(1,1,z)-2(1-z^{n+1})+n(1-z^{n+1})^2}{(1-z^{n+1})^2}.
\end{align*}
The claim of the lemma follows from Lemma \ref{lemma:tedious} below.
\end{proof}
\begin{lemma}\label{lemma:tedious}
We have that
\begin{align*}
&2\lpt_{\Pi(C_1)}(1,1,z)-2\left(1-z^{n+1}\right)+n\left(1-z^{n+1}\right)^2\\
&=(1+z)^2\cdot \left((n+1)(1+z+\dots+z^n)^2-\left(1+z^2+\dots+z^{2n}\right)\right).
\end{align*}
\end{lemma}
\begin{proof}
	Using \eqref{eqn:lpt} we have that
	\[
		\lpt_{\Pi(C_1)}(1,1,z)=1+\left(\sum_{k=1}^n (2(n+1)k-1)\cdot \left(z^k+z^{2(n+1)-k}\right)\right)+(2(n+1)^2-2)z^{n+1}.
	\]
	Thus, the coefficients of $z^k$ on the left-hand side of the claimed equality in the statement of the lemma are symmetric around $z^{n+1}$ and are equal to
\begin{equation*}
  \begin{cases}n& k=0,\\
    4(n+1)k-2& 1\leq k \leq n,\\
4(n+1)^2-2-2n& k=n+1.
\end{cases}
  \end{equation*}
	The coefficients of the expansion of the right-hand side are clearly also symmetric around $z^{n+1}$. 
It is straightforward to expand the right-hand side and compare coefficients with the above.
\end{proof}

To determine a formula for $\chi(n,m)$, we will extract coefficients from its generating function.
We note that
\[
\frac{z\cdot
	\left( (n+1)\left(1+z+\dots+z^{n}\right)^2-\left(1+z^2+\dots+z^{2n}\right) \right)
}{(1-z)^2(1-z^{n+1})^2}=f(z)-g(z)\]
for
\[
	f(z)=\frac{(n+1)z}{(1-z)^4},\quad g(z)=\frac{z\cdot
	\left(1+z^2+\dots+z^{2n}\right) 
}{(1-z)^2(1-z^{n+1})^2}.
\]
\begin{lemma}\label{lemma:coeff}
There is an expansion of $g(z)$ as
\[g(z)=\frac{a_1z+\dots+a_{4n+1}z^{4n+1}
}{(1-z^{n+1})^4}
\]
for some coefficients $a_1,\ldots,a_{4n+1}$.
Set additionally $a_0=a_{4n+2}=a_{4n+3}=0$. Then for $q=0,\ldots,n$, we have
\begin{align*}
	a_q+a_{(n+1)+q}+a_{2(n+1)+q}+a_{3(n+1)+q}&=(n+1)^2,\\
	2a_q+a_{(n+1)+q}-a_{3(n+1)+q}&=(n+1)(q+1),\\
	11a_q+2a_{(n+1)+q}-1a_{2(n+1)+q}+2a_{3(n+1)+q}&=\begin{cases}
		\frac{1}{2}(n+1)^2+3q(q+2) & n\ \textrm{odd}, q\ \textrm{even},\\
		\frac{1}{2}(n+1)^2+3q(q+2)+3 & n\ \textrm{odd}, q\ \textrm{odd},\\
		\frac{1}{2}(n+1)^2+3q(q+2)+\frac{3}{2} & n\ \textrm{even},
	\end{cases}\\
	a_q&=\begin{cases}
		\frac{q(q+2)}{4} & q\ \textrm{even},\\
		\frac{(q+1)^2}{4} & q\ \textrm{odd}.\\
	\end{cases}
\end{align*}

\end{lemma}
\begin{proof}
	The expansion is obtained by multiplying numerator and denominator of $g(z)$ by $(1+z+z^2+\dots+z^n)^2$. Doing so we obtain
\begin{equation}\label{eqn:g}
\left(z+z^3+\dots+z^{2n+1}\right)(1+z+z^2+\dots+z^n)^2=a_1z+\dots+a_{4n+1}z^{4n+1}.
\end{equation}

To compute the coefficients in the expansion of the left-hand side of \eqref{eqn:g}, we consider an $n\times (4n+1)$ array. The columns are labelled by $1,2,\ldots,4n+1$. The first row consists of the entries $1,2,3,\ldots,n+1,n,\ldots,2,1$, followed by zeros. More generally, the $\supth{i}$ row has non-zero entries obtained by shifting the non-zero entries of the first row $2i-2$ positions to the right. See Figure \ref{fig:coeff} for the examples $n=5$ and $n=6$.
Since the coefficients of $(1+z+z^2+\dots+z^n)^2$ are exactly the non-zero entries of the first row of the array, the coefficient $a_i$ is the sum of the entries of the $\supth{i}$ column of the array.

When $n$ is even, we see by inspection that for $q=0,\ldots,n$,
\[
a_q+a_{(n+1)+q}+a_{2(n+1)+q}+a_{3(n+1)+q}=1+2+3+\dots+(n+1)+n+\dots+1.
\]
Similarly, when $n$ is odd, for $q\leq n$ with $q$ even 
\[
a_q+a_{(n+1)+q}+a_{2(n+1)+q}+a_{3(n+1)+q}=2\cdot(2+4+\dots+(n+1)+(n-1)+\dots+2), 
\]
and for $q$ odd we instead have
\[
a_q+a_{(n+1)+q}+a_{2(n+1)+q}+a_{3(n+1)+q}=2\cdot(1+3+\dots+n+n+(n-1)+\dots+1).
\]
All three of these quantities evaluate to $(n+1)^2$. This shows the first desired identity.

For $i=0,\ldots,n+1$ we have by inspection
\[
	a_i=\begin{cases}
		\sum_{j=1}^{i/2} 2j=\frac{i(i+2)}{4} & i\ \textrm{even},\\[1ex]
		\sum_{j=1}^{(i+1)/2} (2j-1)=\frac{(i+1)^2}{4} & i\ \textrm{odd}.\\
	\end{cases}
\]
In particular, this implies the fourth identity.

We next consider the quantity $a_{(n+1)+q}-a_{3(n+1)+q}$ for $0\leq q \leq n$. This is the sum of the first $q+1$ entries in column $(n+1)+q$ and has the form
\[
	a_{(n+1)+q}-a_{3(n+1)+q}=\begin{cases}
		(n+1)+2\sum_{j=1}^{q/2} (n+1-2j) & q\ \textrm{even},\\[1ex]
		2\sum_{j=1}^{(q+1)/2} (n+2-2j) & q\ \textrm{odd}. \\
	\end{cases}
\]
Considering instead $2a_q+a_{(n+1)+q}-a_{3(n+1)+q}$, we obtain
\[
		(n+1)+2\sum_{j=1}^{q/2} ((n+1-2j)+2j)=(n+1)(q+1)
	\]
	for $q$ even and 
\[
		2\sum_{j=1}^{(q+1)/2} ((n+2-2j)+2j-1)=(n+1)(q+1)
	\]
for $q$ odd, proving the second identity.

For the coefficients $a_{2(n+1)+i}$ for $i\geq 0$, we have
\[
	a_{2(n+1)+i}=a_{2n-i}.
\]
We thus obtain
\begin{align*}
	11a_q&+2a_{(n+1)+q}-1a_{2(n+1)+q}+2a_{3(n+1)+q}\\
	&  = 6a_q+6a_{3(n+1)+q}
	-(a_q+a_{(n+1)+q}+a_{2(n+1)+q}+a_{3(n+1)+q})\\
	&\hphantom{=}\;+3(2a_q+a_{(n+1)+q}-a_{3(n+1)+q}) \\
	&= 6(a_q+a_{n-q-1})-(n+1)^2+3(n+1)(q+1).
\end{align*}
Using the above formula for $a_i$ ($i\leq n+1$) and substituting, one obtains the third identity.
\end{proof}

\begin{figure}
	{\scriptsize{\[
	\begin{array}{c c c c c | c c c c c c | c c c c c c | c c c c c c}
		1&2&3&4&5&6&7&8&9&10&11&12&13&14&15&16&17&18&19&20&21
		\\
	\hline	
		1&2&3&4&5&6&5&4&3&2&1&&&&&&\\
		&&1&2&3&4&5&6&5&4&3&2&1&&&&\\
		&&&&1&2&3&4&5&6&5&4&3&2&1&&&\\
		&&&&&&1&2&3&4&5&6&5&4&3&2&1&&\\
		&&&&&&&&1&2&3&4&5&6&5&4&3&2&1\\
		&&&&&&&&&&1&2&3&4&5&6&5&4&3&2&1\\
	\end{array}
\]
}}
\[
n=5
\]
	
\vspace{.5cm}
{\scriptsize{\[
	\begin{array}{c c c c c c | c c c c c c c | c c c c c c c | c c c c c c c}
		1&2&3&4&5&6&7&8&9&10&11&12&13&14&15&16&17&18&19&20&21&22&23&24&25
		\\
	\hline	
		1&2&3&4&5&6&7&6&5&4&3&2&1&&&&&&&\\
		&&1&2&3&4&5&6&7&6&5&4&3&2&1&&&&&&&\\
		&&&&1&2&3&4&5&6&7&6&5&4&3&2&1&&&&&\\
		&&&&&&1&2&3&4&5&6&7&6&5&4&3&2&1&&&\\
		&&&&&&&&1&2&3&4&5&6&7&6&5&4&3&2&1&&\\
		&&&&&&&&&&1&2&3&4&5&6&7&6&5&4&3&2&1\\
		&&&&&&&&&&&&1&2&3&4&5&6&7&6&5&4&3&2&1\\
	\end{array}
\]
}}
\[
n=6
\]
\caption{Example arrays from the proof of Lemma \ref{lemma:coeff}}\label{fig:coeff}
\end{figure}

\subsection{Proof of Theorem \ref{T:chiloc}}\label{proof:chiloc}
	We extract the coefficients in front of $z^{m}$ in the power series $f(z)$ and $g(z)$.
	For 
	\[f(z)=(n+1)z\cdot \left(\sum_{i\geq 0} z^i\right)^4
	\]
this coefficient extraction $[z^m]f(z)$ is straightforward, and we obtain
\begin{align*}
	[z^m]f(z)&=(n+1)\cdot \binom{m+2}{3}=\frac{(n+1)(m+2)(m+1)m}{6}\\
	&=\frac{(n+1)}{6}m^3+\frac{(n+1)}{2}m^2+\frac{(n+1)}{3}m.
\end{align*}

For $[z^m]g(z)$ we use the form of $g(z)$ from Lemma \ref{lemma:coeff} and obtain that $g(z)$ is equal to
\begin{align*}
\sum_{\substack{k\geq 0\\q=0,\ldots,n}}
\left(a_q\binom{k+3}{3}+a_{(n+1)+q}\binom{k+2}{3}+a_{2(n+1)+q}\binom{k+1}{3}+a_{3(n+1)+q}\binom{k}{3}\right)\cdot z^{k(n+1)+q}.
\end{align*}
For $m=k(n+1)+q$ with $q=0,\ldots,n$, and setting $p=n+1$ to simplify notation, it follows that
\begin{align*}
	[z^m]g(z)=a_q\binom{\frac{m-q}{p}+3}{3}+a_{p+q}\binom{\frac{m-q}{p}+2}{3}+a_{2p+q}\binom{\frac{m-q}{p}+1}{3}+a_{3p+q}\binom{\frac{m-q}{p}}{3}.
\end{align*}
We now expand as a polynomial in $m$ to obtain that $[z^m]g(z)$ is
\begin{align*}
	&\frac{1}{6p^{3}}\left(a_q+a_{p+q}+a_{2p+q}+a_{3p+q}\right)m^{3}
\\
&+\frac{1}{2p^3}\left(p\left(2a_q+a_{p+q}-a_{3p+q}\right)-q\left(a_q+a_{p+q}+a_{2p+q}+a_{3p+q}\right)\right)
     m^{2}
     \\
     &\begin{aligned}
        +\frac{1}{6p^3}\left(p^2\left(11a_q+2a_{p+q}-1a_{2p+q}+2a_{3p+q}\right)\right.-6qp\left(2a_q+a_{p+q}-a_{3p+q}\right)&\\
        \left.+3q^2\left(a_q+a_{p+q}+a_{2p+q}+a_{3p+q}\right)\right)&
        m\end{aligned}
      \\
      &\begin{aligned}
      +\frac{1}{6p^3}\left(6p^3a_q
      -qp^2\left(11a_q+2a_{p+q}-1a_{2p+q}\right.+2a_{3p+q}\right)
      +3p^3\left(2a_q+a_{p+q}-a_{3p+q}\right)&\\
	      \left.-q^3\left(a_q+a_{p+q}+a_{2p+q}+a_{3p+q}\right) \right)&.
        \end{aligned}
\end{align*}
Setting
\begin{align*}
	\alpha_1&=a_q+a_{(n+1)+q}+a_{2(n+1)+q}+a_{3(n+1)+q},\\
	\alpha_2&=2a_q+a_{(n+1)+q}-a_{3(n+1)+q},\\
	\alpha_3&=11a_q+2a_{(n+1)+q}-1a_{2(n+1)+q}+2a_{3(n+1)+q},
\end{align*}
we thus have
\begin{align*}
[z^m]g(z)
	=&\frac{1}{6\,p^{3}}\alpha_1 m^{3}
+\frac{1}{2p^3}\left(p\alpha_2-q\alpha_1\right)
     m^{2}
          +\frac{1}{6p^3}\left(p^2\alpha_3-6qp\alpha_2+3q^2\alpha_1\right)
     m\\
      &+\frac{1}{6p^3}\left(6p^3a_q
	      -qp^2\alpha_3
	      +3p^3\alpha_2
	      -q^3\alpha_1\right). 
\end{align*}

Using Lemma \ref{lemma:coeff} to substitute in for $a_q,\alpha_1,\alpha_2,\alpha_3$ and simplifying, we obtain that $[z^m]f(z)-[z^m]g(z)$ is exactly the quasi-polynomial appearing in the statement of Theorem \ref{T:chiloc}. The claim of the theorem thus follows from Lemma \ref{lemma:gen}.\qed

\section{Computation of \texorpdfstring{$\boldsymbol{\chi^0}$}{chi\textasciicircum 0}}
\subsection{A combinatorial formula}\label{sec:chi0formula}
Let $X=X_{\sigma_{A_n}}$ be the toric variety as described in  Example \ref{ex:An}, and let $Y=X_\Sigma$ with $\phi:Y\rightarrow X$ be the minimal resolution, where $\Sigma$ is the fan defined in Example \ref{ex:Anres}. The exceptional locus $E$ consists exactly of torus-invariant divisors $E_1=D_{\rho_1},\ldots,E_n=D_{\rho_n}$.
We shorten notation: 
\[V^i_m(u) = V_{{S}^m\Omega_Y^1}^{\rho_i}(\rho_i(u)).\]
We use that
\[\chi^0\left(s_n,S^m\Omega^1_Y\right)=\dim \frac{\H^0\left(Y\setminus E,S^m\Omega_Y^1\right)}{\H^0\left(Y,S^m\Omega_Y^1\right)}=\sum_{u \in M}
\dim \frac{\H^0\left(Y\setminus E,S^m\Omega_Y^1\right)_u}{\H^0\left(Y,S^m\Omega_Y^1\right)_u}.\]
By \eqref{eqn:global} we have
\[ \begin{aligned}
	\H^0\left(Y\setminus E,{S}^m\Omega_Y^1\right)_u &= V^0_m(u)\cap V^{n+1}_m(u),\\
	\H^0\left(Y,{S}^m\Omega_Y^1\right)_u &= \bigcap_{i=0}^{n+1} V^i_m(u).
\end{aligned}\]
Recall that for $u=(u_1,u_2)$ we have $\rho_i(u)=\rho_i(u_1,u_2)=iu_1+u_2$. We adapt some notation from Section~\ref{sec:filtrations}.
\begin{lemma}\label{lemma:intersections}
Let \[
\lambda_m(i)=\begin{cases}
	0 & i\leq -m,\\
	i+m &-m \leq i \leq 1,\\
	m+1 & i \geq  1.
\end{cases}
\]
Then $\dim V^i_m(u)=m+1-\lambda_m(iu_1+u_2)$. Furthermore, these spaces are maximally independent, so for $I\subset\{0,\ldots,n+1\}$ we have
\[\dim \bigcap_{i\in I} V^i_m(u) = \max\left\{0,m+1-\sum_{i\in I} \lambda_m(iu_1+u_2)\right\}.\]
\end{lemma}
\begin{proof}
The dimension result follows from Example~\ref{ex:Scotangent}. Furthermore, $\bigoplus_{m=0}^\infty V_{S^m\Omega_Y^1}$ is isomorphic to a bivariate polynomial ring in two variables, and the $\rho_i^\perp$ consist of linear forms that are pairwise coprime for different $i$. Hence, if the intersection of several of these spaces is not zero, then the codimension of the intersection is the sum of the codimensions of the spaces.
\end{proof}

We use Lemma~\ref{lemma:intersections}  to write
\[\chi^0\left(s_n,S^m\Omega^1_Y\right)=\sum_{u\in M} z_m(u)\]
with
\[
\begin{aligned}
z_m(u)&=\dim \frac{\H^0\left(Y\setminus E,S^m\Omega_Y^1\right)_u}{\H^0\left(Y,S^m\Omega_Y^1\right)_u}\\
&=\min\left\{\max\left\{0,(m+1-\lambda_m(u_2)-\lambda_m((n+1)u_1+u_2))\right\},\sum_{i=1}^n\lambda_m(iu_1+u_2)\right\}.
\end{aligned}
\]
\begin{lemma}\label{L:chi_0_as_pointcount}
With the definitions above, the set
\[\sG_{m,n}=\left\{(u_1,u_2,z)\in\RR^3 : 0 < z \leq z_m(u_1,u_2)\right\}\]
is a bounded half-open non-convex polytope and
\[\chi^0\left(s_n,S^m\Omega^1_Y\right)=\# \left(\sG_{m,n} \cap \ZZ^3\right).\]
Furthermore, $\sG_m$ is stable under the transformation $(u_1,u_2)\mapsto (-u_1,(n+1)u_1+u_2)$.
\end{lemma}
\begin{proof}
It is straightforward to check that $z_m(u_1,u_2)$ is only non-zero on a bounded region, so $\sG_m$ is bounded. It is a (non-convex) polytope because $z_m(u_1,u_2)$ is piecewise linear. Since $z_m(u_1,u_2)$ takes integer values at $(u_1,u_2)\in\ZZ$, we have that the sum $\sum_{(u_1,u_2)\in\ZZ^2} z_m(u_1,u_2)$ is equal to the lattice point count given.

The symmetry is easily verified through the identity
\[z_m(u_1,u_2)=z_m(-u_1,(n+1)u_1+u_2)\qedhere\]
\end{proof}

In Section~\ref{sec42} we give an explicit description of the non-convex polytope $\sG_{m,n}$ as a dilation of a fixed non-convex polytope $\sG_{0,n}$ by a factor of $m+1$.

\subsection{Explicit description of the non-convex polytope $\boldsymbol{\sG_{m,n}}$}\label{sec42}

As it turns out, we get a nicer description of $\sG_{m,n}$ by shifting our coordinates: we set $(u_1,u_2)=(a,b+1)$. We absorb the shift in a new piecewise linear function $\lambda'_{m+1}$ defined by
\[
\lambda'_{m+1}(i)=\lambda_m(i+1)=\begin{cases}
	0 & i\leq -(m+1),\\
	i+m+1 &-(m+1) \leq i \leq 0,\\
	m+1 & i \geq  0.
\end{cases}
\]
We obtain descriptions
\[z_m(a,b)=\min\left\{\max\left\{0,(m+1-\lambda'_{m+1}(b)-\lambda'_{m+1}((n+1)a+b))\right\},\sum_{i=1}^n\lambda'_{m+1}(ia+b)\right\}\]
and
\[\sG_{m,n}=\left\{(a,b,z)\in\RR^3 : 0 < z \leq z_m(a,b)\right\}.\]
The symmetry of the non-convex polytope $\sG_{m,n}$ in these coordinates is under the same transformation $\tau_n=(a,b)\mapsto(-a,(n+1)a+b)$.

\begin{figure}
	\begin{tikzpicture}[xscale=3.5,yscale=1,pin distance=25]
		\coordinate [label=90:$P_0$](P0) at (-1,0);
		\node at (P0) {$\bullet$};
		\coordinate [label=90:$P_1$](P1) at (-1/2,0);
		\node at (P1) {$\bullet$};
		\coordinate [label=90:$P_2$](P2) at (-1/3,0);
		\node at (P2) {$\bullet$};
		\coordinate [label=80:$P_3$](P3) at (-1/4,0);
		\node at (P3) {$\bullet$};
		\coordinate [pin=-140:$Q_1$](Q1) at (-1/3,-1/3);
		\node at (Q1) {$\bullet$};
		\coordinate [pin=-130:$Q_2$](Q2) at (-1/6,-1/2);
		\node at (Q2) {$\bullet$};
		\coordinate [pin=-110:$Q_3$](Q3) at (-1/10,-3/5);
		\node at (Q3) {$\bullet$};
		\coordinate [label=0:$P'_0$](Pp0) at (1,-4);
		\node at (Pp0) {$\bullet$};
		\coordinate [label=-0:$P'_1$](Pp1) at (1/2,-2);
		\node at (Pp1) {$\bullet$};
		\coordinate [label=0:$P'_2$](Pp2) at (1/3,-4/3);
		\node at (Pp2) {$\bullet$};
		\coordinate [label=10:$P'_3$](Pp3) at (1/4,-1);
		\node at (Pp3) {$\bullet$};
		\coordinate [pin=-90:$Q_1'$](Qp1) at (1/3,-5/3);
		\node at (Qp1) {$\bullet$};
		\coordinate [pin=-90:$Q_2'$](Qp2) at (1/6,-7/6);
		\node at (Qp2) {$\bullet$};
		\coordinate [pin=-110:$Q_3'$](Qp3) at (1/10,-1);
		\node at (Qp3) {$\bullet$};
		\coordinate [label=-110:$Z$](Z) at (0,-1);
		\node at (Z) {$\bullet$};
		\draw (P0) -- (P3) -- (Pp3) -- (Pp0) -- (Z) -- cycle;
		\draw (P0) -- (Q1) -- (Q2) -- (Q3) -- (Qp3) -- (Qp2) -- (Qp1) -- (Pp0);
		\draw (P1) -- (Z);
		\draw (P2) -- (Z);
		\draw (P3) -- (Z);
		\draw (Pp1) -- (Z);
		\draw (Pp2) -- (Z);
		\draw (Pp3) -- (Z);
		\draw[->,dashed] (-1.2,0) -- (1.2,0);
		\draw[->,dashed] (0,-4.5) -- (0,0.5);
		\coordinate [label={$b$-axis}] (b_axis) at (0,0.5);
		\coordinate [label=0:{$a$-axis}] (a_axis) at (1.2,0);
		\draw [dotted] (0,-4)--(1,-4);
		\coordinate [label=180:{$b=-4$}] (btick) at (0,-4);
		\draw [dotted] (1,0)--(1,-4);
		\coordinate [label=90:{$a=1$}] (btick) at (1,0);
	\end{tikzpicture}

	\emph{Note:} The $a$-axis is stretched to ease viewing
	\caption{Top view of $\sG_{0,3}$}
	\label{fig:G03}
\end{figure}
Recall that in Section~\ref{sec:intro} we defined the points
\[
\begin{aligned}
P_i&=\left(-\frac{1}{i+1},0,0\right)\quad\text{for } i=0,1,\ldots,n,\\
Q_i&=\left(-\frac{2}{(i+1)(i+2)},-\frac{i}{i+2},\frac{i}{i+2}\right)\quad\text{for } i=0,1,\ldots,n,\\
Z&=(0,-1,0),\\
\end{aligned}
\]
along with the half-open convex polytopes
\begin{equation*}
\begin{aligned}
\sP_i&=\Conv\{P_{i-1},Q_{i-1},P_i,Q_i,Z\}\setminus \Conv\{P_i,Q_i,Z\}\setminus \Conv \{P_{i-1},P_i,Z\}, \\
\sC_n&=\Conv\{P_n,\tau_n(P_n),Q_n,\tau_n(Q_n),Z\}\setminus\Conv\{P_n,\tau_n(P_n),Z\} 
\end{aligned}
\end{equation*}
from \eqref{E:polyhedra}.
For reference we record
\[\begin{aligned}
	P_n'&=\tau_n(P_n)=\left(\frac{1}{n+1},-1,0\right),\\
	Q_n'&=\tau_n(Q_n)=\left(\frac{2}{(n+1)(n+2)},-1,\frac{n}{n+2}\right).
\end{aligned}
\]

\begin{lemma}\label{L:G_0n_convex_decomposition}
The non-convex polytope $\sG_{m,n}$ is the dilation by $m+1$ of\, $\sG_{0,n}$. Furthermore, we have
\[\sG_{0,n}=\sC_n\cup \bigcup_{i=1}^n \sP_i \cup \bigcup_{i=1}^n\tau_n(\sP_i).\]
\end{lemma}
\begin{proof}
The first claim follows by inspecting the definition of $z_m(a,b)$ and the fact that
\[\lambda'_{m+1}((m+1)i)=(m+1)\lambda'_1(i).\]
It remains to describe $\sG_{0,n}$.
The faces spanned by $\{P_{i-1},P_i,Q_{i-1},Q_i\}$ and $\{Q_{i-1},Q_{i},Z\}$ for $i=1,\ldots,n$
can be checked to be linear parts of the graph of $z_m(a,b)$. See Figure~\ref{fig:G03} for an illustration of the configuration for $n=3$. We define the points $P'_i=\tau_n(P_{i})$ and $Q'_i=\tau_n(Q_{i})$.

By symmetry we get that $\{P'_{i-1},P'_i,Q'_{i-1},Q'_i\}$ and $\{Q'_{i-1},Q'_{i},Z\}$ are also faces of the graph. We get two remaining faces $\{P_n,P'_n,Q_n,Q'_n\}$ and $\{Q_n,Q'_n,Z\}$, and outside these we have that $z_m(a,b)$ is identically zero. 
The description of $\sG_{0,n}$ follows.
\end{proof}
\subsection{Proof of Theorem \ref{T:chi0}}\label{S:proof_chi0}
Lemma~\ref{L:chi_0_as_pointcount} expresses $\chi^0(s_n,S^m\Omega^1_Y)$ as a lattice point count in the dilation by $m+1$ of $\sG_{0,n}$. Lemma~\ref{L:G_0n_convex_decomposition} expresses $\sG_{0,n}$ as a disjoint union of convex polytopes. The theorem follows directly from the volume and lattice point counts of those polytopes.

\subsection{Proof of Proposition \ref{P:chi0_asymptotic}}\label{S:proof_chi0_asymptotic}
We consider the half space $H=\{(a,b,z): a\leq 0\}$. It is straightforward to verify that $\sC_{n-1}\cap H \subset (\sP_{n}\cup \sC_{n})\cap H$, so it follows that $\sG_{0,n-1}\cap H \subset \sG_{0,n}\cap H$.

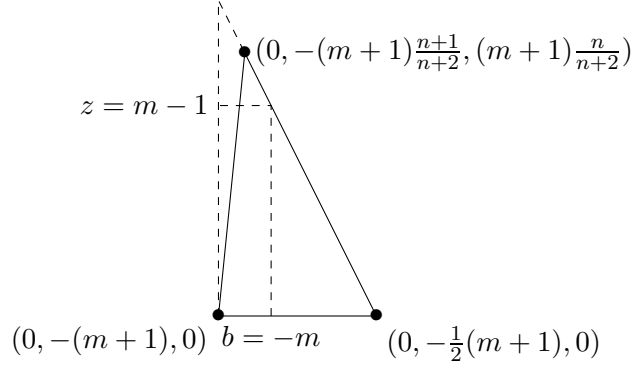
\begin{figure}
	\begin{tikzpicture}[scale=0.7]
		\coordinate [label=-130:{$(0,-(m+1),0)$}] (A) at (-6,0);
		\coordinate [label=-20:{$(0,-\frac{1}{2}(m+1),0)$}](B) at (-3,0);
		\coordinate [label=0:{$(0,-(m+1)\frac{n+1}{n+2},(m+1)\frac{n}{n+2})$}](C) at (-5.5,5);
		\draw [dashed](A) -- (-6,6) -- (C);
		\draw (A) -- (B) -- (C) -- cycle;
		\draw [dashed] (-5,0) -- (-5,4) -- (-6,4);
		\coordinate [label=-90:{$b=-m$}] (xtick) at (-5,0);
		\coordinate [label=-180:{$z=m-1$}] (xtick) at (-6,4);
		\node at (A) {$\bullet$};
		\node at (B) {$\bullet$};
		\node at (C) {$\bullet$};
	\end{tikzpicture}
	\caption{Intersection of $(m+1)\sC_n$ with $a=0$}
	\label{fig:triangle}
\end{figure}

\begin{enumerate}
\item First we note that lattice point counts are non-decreasing with increasing dilation, so $\chi^0(s_n,S^m\Omega^1_Y)$ is non-decreasing in $m$. Since $\sG_{n,m}=(\sG_{n,m}\cap H) \cup \tau_n(\sG_{n,m}\cap H)$ and $\tau_n(\ZZ^3)=\ZZ^3$, we see from the observation above that the lattice point count is also non-decreasing in $n$.
\item To establish that $\chi^0(s_n,S^m\Omega^1_Y)$ is constant in $n$ for $n>m$,  we observe that $(m+1)\sP_n$ does not contain lattice points since any point $(a,b,z)\in(m+1)\sP_n$ satisfies $\frac{m+1}{n+1}<a<0$.
  Similarly, any lattice points in $(m+1)\sC_n$ must have $a=0$ and lie in the triangle with vertices
  \[(0,-(m+1),0),\quad\left(0,-\frac{1}{2}(m+1),0\right),\quad\left(0,-(m+1)\frac{n+1}{n+2},(m+1)\frac{n}{n+2}\right).\]
See Figure~\ref{fig:triangle}.
Only the third point depends on $n$. It lies on the line $z=-2b-(m+1)$ and tends to $(0,-(m+1),(m+1))$ as $n\to\infty$.
Thus we see that the smallest $b$-coordinate of a lattice point has $b\geq -m$ and hence $z\leq m-1$. Such points are already contained in $\sC_n$ for $n=m-1$, so as $n$ grows beyond $m$, we see that $\sG_{n,m}\cap H$ does not gain more lattice points, and therefore the lattice point count in $\sG_{n,m}$ stabilizes for $n\geq m$ as well. 
\item 
A straightforward computation yields
\[\vol \sP_n=\frac{n^2+3n-2}{6n(n+1)^2(n+2)}.\]	
Since
\[\vol\sC_n=\frac{n(n+4)}{6(n+1)(n+2)^2}\]
tends to $0$ as $n\to\infty$, we see that the volume of $\sG_{0,n}$ tends to
\[\lim_{n\to\infty}\sG_{0,n}=2\sum_{n=1}^\infty \vol\sP_n=2\sum_{n=1}^\infty\frac{n^2+3n-2}{6n(n+1)^2(n+2)}=\tfrac{2}{9}\pi^2-2.\]
By Ehrhart theory, we have that $\#(m+1)\sG_{0,n}\cap\ZZ^3$ is a quasi-polynomial in $m$ of degree equal to $\dim \sG_{0,n}=3$, with leading coefficient equal to $\vol \sG_{0,n}$. The argument above establishes that $\vol \sG_{0,n}$ increases with $n$ and tends to $\tfrac{2}{9}\pi^2-2$. The statement follows.
\end{enumerate}

\section{Explicit computation of regular differentials}
\subsection{Setup}
Let $Y$ be a normal surface over $\k$ with function field $\k(Y)$. We write $\Omega_{\k(Y)/\k}$ for the $\k(Y)$-module of K\"ahler differentials.
For any open $U\subset Y$ we have an injection $\H^0(U,S^m\Omega^1_Y)\to S^m\Omega_{\k(Y)/\k}$. We represent a section by its corresponding K\"ahler differential.

The local rings $\sO_{Y,D}$ of prime divisors $D$ on $Y$ give rise to discrete valuations $\ord_D$ on $\k(Y)$. In this section we use the following notation. Given a prime divisor $D$ we choose non-constant functions $\pi,u\in\sO_{Y,D}$ such that $\ord_D(\pi)=1$ and $\ord_D(u)=0$. The differentials $d\pi,du$ form an $\sO_{Y,D}$-basis for $\Omega_{\sO_{Y,D}/\k}$ and therefore also a $\k(Y)$-basis for $\Omega_{\k(Y)/\k}$.

The natural homomorphism $S^m \Omega_{\sO_{Y,D}/\k}\to S^m \Omega_{\k(Y)/\k}$ is an injection, and its image is formed by the differentials that are regular at the generic point of $D$. We define $\ord_D(\omega)$ to be the largest integer $n$ such that $\pi^{-n}\omega \in \Omega_{\sO_{Y,D}/\k}$. 

For $\omega\in S^m\Omega_{\k(Y)/\k}$ we have
$\omega=f_0 (du)^m+f_1(du)^{m-1}d\pi+\cdots+f_m (d\pi)^m$ and
\[\ord_D \omega=\min\left\{\ord_D(f_i): i=0,\ldots,m\right\}.\]
We furthermore have a reduction homomorphism $\rho_D\colon S^m\Omega_{\sO_{Y,D}/\k}\to S^m\Omega_{\k(D)/\k}$ by reducing modulo $\pi$ and sending $d\pi$ to $0$.
\begin{definition}
Let $Y$ be as above, and let $\omega\in S^m\Omega_{\k(Y)/k}$ be non-zero. We say that a prime divisor $D\subset Y$ is a \emph{solution curve} to $\omega$ if $\rho(\pi^{-\ord_D(\omega)}\omega)=0$.
\end{definition}
In terms of the coordinates described above, $D$ is a solution curve to $\omega$ if and only if $\ord_D(f_0) > \ord_D(\omega)$.

\begin{proposition}\label{P:ord_solution_curve}
	Let $\psi\colon Z\to Y$ be a finite morphism of normal surfaces. Let $D\subset Y$ be a prime divisor, and let $D'\subset Z$ be a prime divisor above $D$ of ramification degree $e$. Suppose that $D$ is a solution curve to $\omega\in S^m\Omega_{\k(Y)/\k}$. Then
\[\ord_{D'}\psi^*\omega \geq e\ord_{D}\omega + (e-1).\]
\end{proposition}
\begin{proof}
	The inequality is preserved under scaling $\omega$ by a non-zero element of $\k(Y)$, so it suffices to deal with the case $\ord_D(\omega)=0$. Let us take a uniformizer $\pi$ at $D$. Identifying $\k(Y)\subset \k(Z)$ via the pull-back $\psi^*$, we have that a uniformizer $\pi'$ at $D'$ is of the form $\pi=v(\pi')^e$ for some $v\in \k(Z)$ with $\ord_{D'}(v)=0$. We also choose a non-constant $u\in \k(Y)$ so that we have
\[\omega=f_0 (du)^m+f_1(du)^{m-1}d\pi+\cdots+f_m (d\pi)^m\]
with $\min \ord_D(f_i)=0$. The fact that $D$ is a solution curve means that $\ord_D(f_0)\geq 1$ and therefore $\ord_{D'}(f_0)\geq e$.

Note that $d \pi = d(v(\pi')^e) = (\pi')^e dv + ev(\pi')^{e-1}d\pi'$, so we have
\[\ord_{D'} \left(f_i (du)^{m-i}(d\pi)^i\right) \geq e\ord_D(f_i) + e-1 \quad\text{for }i=1,\ldots,m.\]
This implies that $\ord_{D'}(\psi^*\omega)\geq e-1$, as required.
\end{proof}

Let us now consider a normal surface $X$, with singular locus $S$ and  minimal resolution $Y$. Then $\k(X)$ and $\k(Y)$ are canonically isomorphic. 
Let $E\subset Y$ be the locus of $Y$ mapping to $S$ on $X$. Then $X\setminus S$ is isomorphic to $Y\setminus E$.

Because $S$ is of codimension $2$ in $X$, we can extend the sheaf $S^m\Omega^1_{X\setminus S}$ uniquely to a reflexive sheaf $\hat S^m\Omega^1_X$ on $X$, and its sections are completely determined by their behaviour on $X\setminus S$. Here too, we represent sections by the corresponding K\"ahler differentials: a section is regular on $X$ if it is regular at all divisors on~$Y$ that are not contained in $E$.

Now suppose that we have a differential $\omega$ that is regular on $X\setminus S$. If $s\in S$ is an $A_n$-singularity, then we can bound $\ord_{E_i}\omega$ at components above $s$ as well.

\begin{proposition}\label{P:ord_bound_quotient_singularity}
Let $X$ be a surface with an $A_n$-singularity $s$, let $Y$ be a minimal resolution of\, $X$, and let $E_i\subset Y$ be a prime divisor of\, $Y$ above $s$. Suppose that $\omega\in \H^0(X,\hat S^m\Omega_X^1)$. Then
\[\ord_{E_i} \omega \geq \left\lceil\frac{-mn}{n+1}\right \rceil.\]
\end{proposition}
\begin{proof}
We use the notation from Example \ref{ex:Anres}. Then $E_i=D_{\rho_i}$ for some $i\in\{1,\ldots,n\}$.
It suffices to show the claim for torus semi-invariant symmetric differentials, so let $\omega \in \H^0(X,\hat S^m\Omega^1_X)_u$ for some weight $u\in \ZZ^2$.
In the notation of Section~\ref{sec:chi0formula}, we have $V_m^0(u)\cap V_m^{n+1}(u)\neq 0$. By Lemma \ref{lemma:intersections} we have
\[
	\rho_0(u)\leq 0, \quad \rho_{n+1}(u)\leq 0,\quad \rho_0(u)+\rho_{n+1}(u)\leq -m.
\]
For the ray $\rho_i$, $1\leq i \leq n$, we have 
\[
	\rho_i(u)= \frac{1}{n+1}\left((n+1-i)\rho_0(u)+i\rho_{n+1}(u)\right).
\]
Since both $(n+1-i)$ and $i$ are at least $1$, it follows that $\rho_i(u) \leq -m/(n+1)$, or equivalently,
\[
	j=\rho_i(u)+\left \lceil \frac{m\cdot (-n)}{n+1}\right \rceil \leq -m.
\]
From Example~\ref{ex:Scotangent} it follows that $\omega\in x^u V^{\rho_i}(j)$ and hence by \eqref{E:shifted_filtration} that
\[\ord_{D_{\rho_i}}(\omega)\geq j-\rho_i(u)=\left \lceil \frac{m\cdot (-n)}{n+1}\right\rceil.\qedhere\]
\end{proof}

\subsection{Labs surfaces}\label{S:labs}
We consider surfaces in $\PP^3$ from Labs \cite[Corollary~A1]{Labs2006} of degree $d=2k$, constructed by taking a smooth plane quadric $X_1\subset\PP^3$ tangent to the four coordinate planes and pulling it back to a surface $X_k$ under the map $(\xi_0:\xi_1:\xi_2:\xi_3)\mapsto(\xi_0^k:\xi_1^k:\xi_2^k:\xi_3^k)$. An explicit choice of model yields
\[X_k\colon \xi_0^{2k}+\xi_1^{2k}+\xi_2^{2k}+\xi_3^{2k}
-\xi_0^k\xi_1^k-\xi_0^k\xi_2^k-\xi_0^k\xi_3^k-\xi_1^k\xi_2^k-\xi_1^k\xi_3^k-\xi_2^k\xi_3^k=0.\]
We write $S$ for the singular locus of $X_k$. It consists of $4k^2$ isolated singularities of type $A_{k-1}$. The singularities lie in the four planes $\xi_0=0,\ldots,\xi_3=0$. Writing $\zeta_k$ for a primitive $k$\textsuperscript{th} root of unity, the singularities with $\xi_3=0$ have coordinates
\[(1:\zeta_k^i:\zeta_k^j:0) \quad\text{for } i,j=0,\ldots,k-1.\] 

We observe that $X_1$ is a non-singular quadric and that we have a finite morphism $\phi_k\colon X_k\to X_1$ defined by $(\xi_0:\xi_1:\xi_2:\xi_3)\mapsto(\xi_0^k:\xi_1^k:\xi_2^k:\xi_3^k)$ of degree $k^3$ and branch locus $\xi_0\xi_1\xi_2\xi_3=0$, of ramification degree $k$ over each of those plane sections.

Writing $\zeta_3$ for a primitive third root of unity, we have that $X_1$ is isomorphic to $\PP^1\times\PP^1$ over a field containing $\zeta_3$. In terms of affine coordinates $(s_0:s_1)\times(t_0:t_1)$, we can express the isomorphism as
\[\begin{aligned}
\xi_0&=3s_0t_0,\\
\xi_1&=s_1t_1+(\zeta_3+2)s_1t_0-(\zeta_3-1)s_0t_1+3s_0t_0,\\
\xi_2&=s_1t_1+(2\zeta_3+1)s_1t_0-(2\zeta_3+1)s_0t_1+3s_0t_0,\\
\xi_3&=s_1t_1.
\end{aligned}\]
We note that the plane $\xi_3=0$ is tangent to $X_1$ and hence that it intersects $X_1$ in two lines $L_{3,1}, L_{3,2}$. By symmetry, the same holds for the other coordinate planes $\xi_0=0$, $\xi_1=0$, $\xi_2=0$, for which we adopt the same notation.

We pass to affine coordinates $(x_1,x_2,x_3)=(\xi_1/\xi_0,\xi_2/\xi_0,\xi_3/\xi_0)$ on $\PP^3$ and
$(s,t)=(s_1/s_0,t_1/t_0)$ on $\PP^1\times \PP^1$. We obtain
\[\begin{aligned}
3dx_1&=(t+\zeta_3+2)ds+(s-\zeta_3+1)dt,\\
3dx_2&=(t+2\zeta_3+1)ds+(s-2\zeta_3-1)dt,\\
3dx_3&=tds+sdt.
\end{aligned}\]
We consider the degree $2$ differential $dsdt$ on $\PP^1\times\PP^1$, which under the isomorphism above yields
\[\omega_1 = \frac{-3}{x_1^2+x_2^2+1-2x_1x_2-2x_1-2x_2}\left(x_2(dx_1)^2+(1-x_1-x_2)dx_1dx_2+x_1(dx_2)^2\right).\]
We record a few facts about $\omega_1$.
\begin{lemma}\label{L:solutions_X1}
The differential $\omega_1$ has $\ord_{L_{0,1}}\omega_1=\ord_{L_{0,2}}\omega_1=-2$ and is regular elsewhere. Furthermore, the solution curves to $\omega_1$ are exactly the lines constituting the two rulings on $X_1$.
\end{lemma}
\begin{proof}
On $\PP^1\times\PP^1$ we easily see that $dsdt$ has double poles at $s=\infty$ and $t=\infty$ and nowhere else. Furthermore, it is straightforward to check that the two rulings on $\PP^1\times\PP^1$ 
form exactly the solution curves of $dsdt=0$.
The statement on $X_1$ follows simply by applying the isomorphism $\PP^1\times\PP^1\to X_1$.
\end{proof}

We next consider $\phi_k\colon X_k\to X_1$. The inverse images of the lines $L_{0,i},L_{1,i},L_{2,i},L_{3,i}$ are prime divisors $D_{0,i},D_{1,i},D_{2,i},D_{3,i}$. These are degree $k$ Fermat curves, as one can see from the factorization
\[\xi_0^{2k}+\xi_1^{2k}+\xi_2^{2k}-\xi_0^k\xi_1^k-\xi_0^k\xi_2^k-\xi_1^k\xi_2^k=\left(\xi_0^k+\zeta_3 \xi_1^k+\zeta_3^2\xi_2^k\right)\left(\xi_0^k+\zeta_3^2 \xi_1^k+\zeta_3\xi_2^k\right).\]

We consider the pull-back $\omega_k=\phi_k^*\omega_1$ to $X_k$.

\begin{lemma}\label{L:S2hat_global_section}
For $D=D_{0,i}$ we have $\ord_{D}\omega_k \geq -k-1$, and $\omega_k$, as a section of\, $\hat S^2\Omega^1_{X_k}$, is regular elsewhere. Furthermore, for $D=D_{1,i},D_{2,i},D_{3,i}$ we have $\ord_{D}\omega_k \geq k-1$.
	
As a result, for $k\geq 2$ we have that 
$(x_1x_2x_3)^{1-k}\omega_k$
is a global section of $\hat S^2\Omega^1_{X_k}$, and for $k\geq 4$ we have that
$\tilde\omega_k=(x_1x_2x_3)^{2-k}\omega_k$
is a global section that vanishes identically on $\xi_0\xi_1\xi_2\xi_3=0$.
\end{lemma}
\begin{proof}
The curves mentioned in the lemma lie over solution curves for $\omega_1$ with ramification degree $k$. The first claims are a direct application of Proposition~\ref{P:ord_solution_curve}.

The second part is just the observation that $\xi_0,\xi_1,\xi_2,\xi_3$ vanish to the first order on their respective curves.
\end{proof}

\begin{lemma}\label{L:genus_of_solution_curves}
	The solution curves of $\tilde\omega_k$ contained in $\xi_0\xi_1\xi_2\xi_3=0$ are degree $k$ non-singular plane curves and hence of genus $(k-1)(k-2)/2$. The other solution curves are non-singular complete intersections of two degree $k$ surfaces and hence curves of genus $k^3-2k^2+1$.
\end{lemma}
\begin{proof}
By Lemma~\ref{L:solutions_X1} we see that the solution curves arise as fibres of the composition $X^k\to X_1\to \PP^1$ induced by the projections from $X_1\simeq \PP^1\times\PP^1$ onto either of the factors.

Let us first consider the projection onto the first factor. The fibre over the point $(1:s)$ can be expressed as an intersection of planes on $X_1$. Computation shows it is the kernel of the matrix
\[A=\begin{pmatrix}
-(\frac{1}{3}(\zeta_3+2)s^2+s)&s&0&-(s-\zeta_3+1)\\
-(\frac{1}{3}(2\zeta_3+1)s^2+s)&0&s&-(s-2\zeta_3-1)\\
\end{pmatrix},\]
so the corresponding solution curve on $X_k$ is described by
\[
\begin{cases}
A_{1,0}\xi_0^k+A_{1,1}\xi_1^k+A_{1,2}\xi_2^k+A_{1,3}\xi_3^k=0,\\
A_{2,0}\xi_0^k+A_{2,1}\xi_1^k+A_{2,2}\xi_2^k+A_{2,3}\xi_3^k=0.
\end{cases}
\]
Using the Jacobian criterion, any singular point must have an appropriate mixture of vanishing of homogeneous coordinates and minors of $A$.
However, those minors only consist of factors $s,(s-\zeta_3+1),(s-2\zeta_3-1)$, which lead to the curves $D_{1,i},D_{2,i},D_{3,i}$ contained in $\xi_1\xi_2\xi_3=0$. For $s=\infty$ we obtain the curves $D_{0,i}$ contained in $\xi_0=0$.

The other ruling consists of fibres over points $(1:t)$ on the second factor and behaves symmetrically to this one.
\end{proof}

\subsection{Proof of Theorem~\ref{T:labs_hyperbolic}}\label{S:labs_hyperbolic}
Let $Y_k\to X_k$ be a minimal desingularization of $X_k$. By Lemma~\ref{L:S2hat_global_section} we have that $\tilde\omega_k$ is a regular section of $\hat S^2\Omega_X^1$. We consider the pull-back of $\tilde\omega_k$ to $Y_k$, and we also denote it by $\tilde\omega_k$. From Lemma~\ref{L:S2hat_global_section} we obtain that $\tilde\omega_k$ is regular outside any prime divisor $D$ of $Y_k$ above a singularity $s$ of $X_k$. In fact, since $s$ lies on $\xi_0\xi_1\xi_2\xi_3$, which is contained in the vanishing locus of $\tilde\omega_k$, we see that $\tilde\omega_k=f\omega'$ for some differential $\omega'$ regular around $s$ and function $f$ vanishing at $s$. We identify $f$ with its pull-back to $Y_k$ and conclude that $\ord_D(f)\geq 1$.

The singularities on $X_k$ are of type $A_{k-1}$, so Proposition~\ref{P:ord_bound_quotient_singularity} yields that $\ord_D \omega' \geq \lceil \frac{2(1-k)}{k}\rceil =-1$.  It follows that $\ord_D(\tilde\omega_k)=\ord_D(f)+\ord_D(\omega')\geq 0$, so $\tilde\omega_k$ is regular everywhere on $Y_k$. 

For a curve $D\subset Y_k$ we obtain a pull-back $S^m\Omega^1_Y\to S^m\Omega^1_D$ that preserves regularity. Hence, $\tilde\omega_k$ pulls back to a regular form on $D$. On a genus $0$ curve such forms must vanish, so such a $D$ must be a solution curve to $\tilde\omega_k$. Apart from the exceptional components above singularities, any such curve is a pull-back of a solution curve to $\tilde\omega_k$ on $X_k$. By Lemma~\ref{L:genus_of_solution_curves} none of these are of genus $0$.

For $k>5$ we have more freedom: from valuations we can conclude that \[\eta=(a_0+a_1x_1+a_2x_2+a_3x_3)\tilde\omega_k\] represents a regular differential on $Y_k$ that vanishes on $a_0\xi_0+a_1\xi_1+a_2\xi_2+a_3\xi_3=0$. 

For a putative genus $1$ curve $C$ on $X_k$, one can then choose a plane 
$a_0\xi_0+a_1\xi_1+a_2\xi_2+a_3\xi_3=0$ that intersects $C$ transversally. Reducing to $C$ would yield a regular degree $2$ differential on $C$ that additionally has zeros on $C$. But then $\eta$ must reduce to $0$ on $C$; \textit{i.e.}~$C$ is a solution curve to $\eta$. Since $\eta$ is a scaling of $\tilde\omega_k$ by an element in $\k(Y)$, the two differentials have the same solution curves. No curves on $Y_k$ above singularities of $X_k$ can be of genus $1$, so $C$ would need to be a pull-back of a solution curve to $\tilde\omega_k$ on $X_k$. Again, by Lemma~\ref{L:genus_of_solution_curves} such curves do not have genus $1$.

\renewcommand\thesection{\Alph{section}}
\setcounter{section}{0}

\section*{Appendix. Ehrhardt generating functions}
\addcontentsline{toc}{section}{Appendix. Ehrhardt generating functions}
\refstepcounter{section}
\setcounter{subsection}{1}

We give the generating functions of the lattice point counts in dilations $(m+1)\sP_n$ for $n=1,\ldots,5$.
{\scriptsize
\[
\def\arraystretch{2}
\begin{array}{c|l}
n&\textrm{Generating function }\sum_{t=0}^\infty L(\sP_n,m+1)t^m\\
\hline
1 & \frac{t^{3}}{{\left(t^{2} + t + 1\right)} {\left(t + 1\right)} {\left(t - 1\right)}^{4}} \\
2 & \frac{{\left(t^{4} + t^{2} - t + 1\right)} t^{2}}{{\left(t^{2} + t + 1\right)}^{2} {\left(t^{2} - t + 1\right)} {\left(t + 1\right)} {\left(t - 1\right)}^{4}} \\
3 & \frac{{\left(t^{11} + t^{9} + t^{8} + t^{7} + t^{6} + t^{4} + t^{2} + 1\right)} t^{3}}{{\left(t^{4} + t^{3} + t^{2} + t + 1\right)} {\left(t^{4} - t^{3} + t^{2} - t + 1\right)} {\left(t^{2} + t + 1\right)} {\left(t^{2} - t + 1\right)} {\left(t^{2} + 1\right)} {\left(t + 1\right)}^{2} {\left(t - 1\right)}^{4}} \\
4 & \frac{{\left(t^{18} + t^{16} + t^{14} + t^{13} + t^{12} + t^{11} + t^{10} + t^{9} + t^{7} + t^{5} + t^{4} + t^{2} + 1\right)} t^{4}}{{\left(t^{8} - t^{7} + t^{5} - t^{4} + t^{3} - t + 1\right)} {\left(t^{4} + t^{3} + t^{2} + t + 1\right)}^{2} {\left(t^{4} - t^{3} + t^{2} - t + 1\right)} {\left(t^{2} + t + 1\right)} {\left(t^{2} + 1\right)} {\left(t + 1\right)} {\left(t - 1\right)}^{4}} \\
5 & \frac{{\left(t^{28} + t^{25} + t^{23} + t^{22} + t^{20} + t^{19} + t^{18} + t^{17} + t^{16} + t^{15} + t^{14} + t^{12} + t^{11} + t^{9} + t^{8} + t^{6} + t^{5} + t^{3} + 1\right)} t^{5}}{{\left(t^{12} - t^{11} + t^{9} - t^{8} + t^{6} - t^{4} + t^{3} - t + 1\right)} {\left(t^{8} - t^{7} + t^{5} - t^{4} + t^{3} - t + 1\right)} {\left(t^{6} + t^{5} + t^{4} + t^{3} + t^{2} + t + 1\right)} {\left(t^{4} + t^{3} + t^{2} + t + 1\right)} {\left(t^{2} + t + 1\right)}^{2} {\left(t^{2} - t + 1\right)} {\left(t + 1\right)} {\left(t - 1\right)}^{4}} \\
\end{array}
\]
}
We also give the generating functions of the lattice point counts in dilations $(m+1)\sC_n$:
{\scriptsize\[\def\arraystretch{2}
\begin{array}{c|l}
n&\textrm{Generating function }\sum_{t=0}^\infty L(\sC_n,m+1)t^m\\
\hline
1 & \frac{{\left(t^{4} + t^{3} + 2 \, t^{2} + 3 \, t + 3\right)} t^{2}}{{\left(t^{2} + t + 1\right)}^{2} {\left(t + 1\right)}^{2} {\left(t - 1\right)}^{4}} \\
2 & \frac{{\left(t^{4} - t^{2} + 2 \, t + 1\right)} t^{2}}{{\left(t^{2} + t + 1\right)} {\left(t^{2} - t + 1\right)} {\left(t + 1\right)}^{2} {\left(t - 1\right)}^{4}} \\
3 & \frac{{\left(t^{12} + t^{10} + 2 \, t^{8} + 2 \, t^{6} + 2 \, t^{5} + 2 \, t^{4} + 3 \, t^{2} + 1\right)} {\left(t^{2} + t + 1\right)} t^{2}}{{\left(t^{4} + t^{3} + t^{2} + t + 1\right)}^{2} {\left(t^{4} - t^{3} + t^{2} - t + 1\right)} {\left(t^{2} + 1\right)} {\left(t + 1\right)}^{2} {\left(t - 1\right)}^{4}} \\
4 & \frac{{\left(t^{9} + t^{7} - t^{6} + t^{3} + t^{2} + 1\right)} {\left(t^{4} - t^{3} + t^{2} - t + 1\right)} {\left(t + 1\right)} t^{2}}{{\left(t^{8} - t^{7} + t^{5} - t^{4} + t^{3} - t + 1\right)} {\left(t^{4} + t^{3} + t^{2} + t + 1\right)} {\left(t^{2} + t + 1\right)}^{2} {\left(t - 1\right)}^{4}} \\
5 & \frac{{\left(t^{24} + t^{21} + 2 \, t^{18} + 2 \, t^{15} + 2 \, t^{13} + 2 \, t^{12} - 2 \, t^{11} + 2 \, t^{10} + 2 \, t^{9} + 2 \, t^{7} + 2 \, t^{4} + t^{3} + 1\right)} {\left(t^{4} + t^{3} + t^{2} + t + 1\right)} t^{2}}{{\left(t^{12} - t^{11} + t^{9} - t^{8} + t^{6} - t^{4} + t^{3} - t + 1\right)} {\left(t^{6} + t^{5} + t^{4} + t^{3} + t^{2} + t + 1\right)}^{2} {\left(t^{2} + t + 1\right)} {\left(t^{2} - t + 1\right)} {\left(t + 1\right)}^{2} {\left(t - 1\right)}^{4}} \\
\end{array}
\]
}
As an example, for $n=2$ we get the generating function
\begin{align*}
\sum_{m=1}^\infty \chi^0(s_2,S^m\Omega_Y)t^m&=
 2\left(
\frac{t^{3}}{{\left(t^{2} + t + 1\right)} {\left(t + 1\right)} {\left(t - 1\right)}^{4}}
+ \frac{{\left(t^{4} + t^{2} - t + 1\right)} t^{2}}{{\left(t^{2} + t + 1\right)}^{2} {\left(t^{2} - t + 1\right)} {\left(t + 1\right)} {\left(t - 1\right)}^{4}}
\right)\\
&\hphantom{=}\; +\frac{{\left(t^{4} - t^{2} + 2 \, t + 1\right)} t^{2}}{{\left(t^{2} + t + 1\right)} {\left(t^{2} - t + 1\right)} {\left(t + 1\right)}^{2} {\left(t - 1\right)}^{4}}\\
&=3t^2 + 8t^3 + 15t^4 + 28t^5 + O(t^6).
\end{align*}
See the ancillary files \cite{BIXanc} provided with this article for machine-readable representations of the corresponding quasi-polynomials, as well as sample Sage code for generating this data.


\end{document}